\documentclass[reqno,english]{amsart}
\usepackage{amsfonts,amsmath,latexsym,verbatim,amscd,mathrsfs,color,array}

\usepackage{amsmath,amssymb,amsthm,amsfonts,graphicx,color}
\usepackage{amssymb}
\usepackage{pdfsync}
\usepackage{epstopdf}
\usepackage[colorlinks=true]{hyperref}
\usepackage{subfigure}

\makeatletter
\def\@currentlabel{2.1}\label{e:dispaa}
\def\@currentlabel{2.21}\label{e:dispau}
\def\@currentlabel{2.22}\label{e:dispav}
\def\@currentlabel{2.23}\label{e:dispaw}
\def\@currentlabel{2.24}\label{e:dispax}
\def\theequation{\thesection.\@arabic\c@equation}
\makeatother
\let\oldbibliography\thebibliography
\renewcommand{\thebibliography}[1]{%
\oldbibliography{#1}%
\setlength{\itemsep}{0pt}%
}

\renewcommand{\theequation}{\thesection.\arabic{equation}}
\newtheorem{lemma}{Lemma}[section]

\newtheorem{definition}{Definition}[section]

\newtheorem{proposition}{Proposition}[section]
\newtheorem{corollary}{Corollary}[section]
\newtheorem{remark}{Remark}[section]
\newcommand{\bremark}{\begin{remark} \em}
\newcommand{\eremark}{\end{remark} }

\newtheorem*{thmA}{Theorem A}
\newtheorem{open problem}{Open Problem}[section]
\newtheorem{open question}{Open Question}[section]
\newtheorem{open questions on hexagonal lattice}{Open Questions on hexagonal lattice}[section]
\newtheorem{theorem}{Theorem}[section]

\newtheorem*{problemA}{Problem A}
\newtheorem{example}{Example}[section]
\newtheorem*{reformulation}{Reformulation}

\newcommand{\R}{{\mathbb R}}

\newcommand{\BE}{\begin{equation}}
\newcommand{\BEN}{\begin{equation*}}
\newcommand{\EE}{\end{equation}}
\newcommand{\EEN}{\end{equation*}}
\newcommand{\BL}{\begin{lemma}}
\newcommand{\EL}{\end{lemma}}
\newcommand{\BT}{\begin{theorem}}
\newcommand{\ET}{\end{theorem}}
\newcommand{\BP}{\begin{proposition}}
\newcommand{\EP}{\end{proposition}}
\newcommand{\BC}{\begin{corollary}}
\newcommand{\EC}{\end{corollary}}
\renewcommand{\Re}{\operatorname{Re}}
\renewcommand{\Im}{\operatorname{Im}}

\begin{document}


\title[Hexagonal to square phase transition]{\bf On a variational model for the continuous mechanics exhibiting hexagonal to square Phase Transitions}

\author{Senping Luo}

\author{Juncheng Wei}

\address[S.~Luo]{School of Mathematics and statistics, Jiangxi Normal University, Nanchang, 330022, China}
\address[J.~Wei]{Department of Mathematics, University of British Columbia, Vancouver, B.C., Canada, V6T 1Z2}

\email[S.~Luo]{luosp1989@163.com}

\email[J.~Wei]{jcwei@math.ubc.ca}

\begin{abstract}
Inspired by   Conti and Zanzotto \cite{Conti2004A}, we reformulate a simple variational model for reconstructive phase transitions in crystals arising in continuum mechanics in the framework of Landau's theory of phase transition(with slight modification). We provide and prove that this  class of modular invariant functions admit exactly hexagonal-square lattices minimizers without passing through rhombic lattices, being the first rigorous result in this regard.
Our result gives an affirmative answer to an open problem by in \cite{Conti2004A}. In addition, our result has independent interest from number theory.

\end{abstract}

\maketitle

\setcounter{equation}{0}

\section{Introduction and main results}
\setcounter{equation}{0}

We start with some basic analysis in continuum mechanics. For more details, see \cite{Conti2019,Z2020,Z2023} and the references therein. Let $\mathbf{y}=\mathbf{y}(\mathbf{x})$  be the deformation of a continuum body, $\mathbf{y}$ be the current configuration
and $\mathbf{x}$ be the reference state. The strain energy density of an elastic solid $\mathbf{\phi}=\phi(\mathbf{C})$ depends on deformation
gradient $\mathbf{F}=\nabla \mathbf{y}$ through Cauchy-Green tensor $\mathbf{C}=\mathbf{F}'\mathbf{F}$. The strain energy density
$\phi(\mathbf{C})$ admits rotational invariance for $\mathbf{F}\in\mathbf{O}(n)$ by the structure of Cauchy-Green tensor $\mathbf{C}$.
Among all deformations that map a Bravais lattice into itself, the strain energy density $\phi(\mathbf{C})$
must satisfy
\begin{equation}\aligned\label{V1}
\phi(\mathbf{C})=\phi(\mathbf{m}'\mathbf{C}\mathbf{m}),
\endaligned\end{equation}
where $\mathbf{m}$ is taken into $\hbox{GL}(n,\mathbb{Z})$ that denotes the classical modular group and $\det(\mathbf{m})=\pm1$, and the dimension $n=2,3$.
 We are interested in dimension $n=2$, which is closely connected to complex analysis and modular functions on the Poincar\'e upper complex half-plane. To see this, let $\tilde{\mathbf{C}}:=\det(\mathbf{C})^{-\frac{1}{2}}\mathbf{C}$ be the unimodular tensor, one can
 smoothly map the space of unimodular (positive-definite, symmetric) strain tensors $\tilde{\mathbf{C}}$ bijectively to the complex half-plane $\mathbb{H}$: namely,
 \begin{equation}\aligned\label{V2}
\tilde{\mathbf{C}}=\left( \begin{array}{cc}
        \tilde{\mathbf{C}}_{11} & \tilde{\mathbf{C}}_{12} \\
        \tilde{\mathbf{C}}_{21} & \tilde{\mathbf{C}}_{22} \\
      \end{array}\right)
      \mapsto z:= \tilde{\mathbf{C}}_{11}^{-1}( \tilde{\mathbf{C}}_{12}+i)\in\mathbb{H},\\
 \mathbb{H}=\{z= x+ i y\;\hbox{or}\;(x,y)\in\mathbb{C}: y>0\},\;\; ds^2=\frac{dx^2+dy^2}{y^2}
\endaligned\end{equation}
The bijection in \eqref{V2} leads to the strain energy density $\phi$(after changing of variables) satisfies
\begin{equation}\aligned\label{V3}
\phi(\gamma(z))=\phi(z),\;\; \gamma\in \hbox{SL}(2,\mathbb{Z}),
\endaligned\end{equation}
where $\hbox{SL}(2,\mathbb{Z})$ is the classical modular group and denotes the subgroup collecting the positive-determinant elements of
$\hbox{GL}(2,\mathbb{Z})$. Precisely, $\hbox{SL}(2,\mathbb{Z})$ is the group generated by two elements
\begin{equation}\aligned\nonumber
z\mapsto -\frac{1}{z};\;\; z\mapsto z+1; .
\endaligned\end{equation}
The fundamental domain associated to the group $\hbox{SL}(2,\mathbb{Z})$ is
\begin{equation}\aligned\label{Fd30}
\mathcal{D}_0:=\{
z\in\mathbb{H}: |z|\geq1,\; |\Re(z)|\leq\frac{1}{2}
\}.
\endaligned\end{equation}
In this paper, the fundamental domain used is
\begin{equation}\aligned\label{Fd3}
\mathcal{D}:=\{
z\in\mathbb{H}: |z|\geq1,\; 0\leq\Re(z)\leq\frac{1}{2}
\},
\endaligned\end{equation}
induced by the group generated by
\begin{equation}\aligned\nonumber
z\mapsto -\frac{1}{z};\;\; z\mapsto z+1;\;\hbox{and}\; z\mapsto -\overline{z}.
\endaligned\end{equation}

The boundary of the fundamental domain $\mathcal{D}$ (see Picture \ref{LJFFF}) has mechanical meaning and is defined by
\begin{equation}\aligned\nonumber
\Gamma_a:&=\{
z\in\mathbb{H}: z=iy,\; y\geq1
\},\;\;
\Gamma_b:=\{
z\in\mathbb{H}: z=e^{i\theta},\; \theta\in[\frac{\pi}{3},\frac{\pi}{2}]
\},\\
\Gamma_c:&=\{
z\in\mathbb{H}: z=\frac{1}{2}+iy,\; y\geq\frac{\sqrt3}{2}
\},\;\; \partial\mathcal{D}=\Gamma_a\cup\Gamma_b\cup\Gamma_c.
\endaligned\end{equation}
\begin{figure}
\centering
 \includegraphics[scale=0.58]{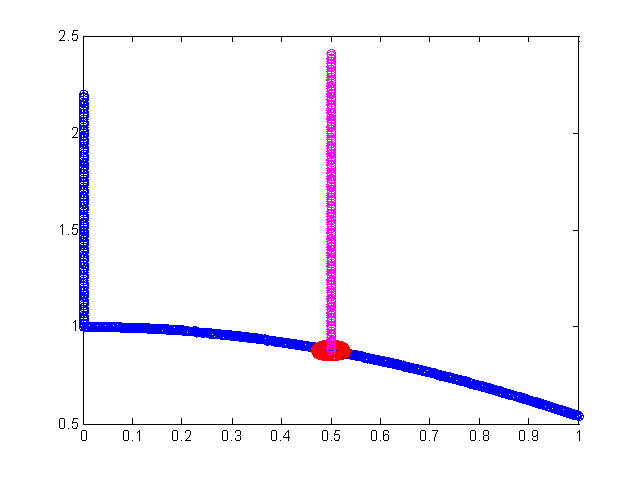}
 \caption{Boundary of the fundamental domain $\mathcal{D}$.}
\label{LJFFF}
\end{figure}

The interior points of $\mathcal{D}$ through \eqref{V2} correspond to oblique lattices with trivial symmetry. Again through \eqref{V2}, the boundary points represent strain tensors and hence lattices in different meaning.
The boundary points on the imaginary axis($z\in\Gamma_a$) correspond to rectangular lattices, the boundary points on the unit arc($z\in\Gamma_b$) correspond to flat-rhombic lattices, and the boundary points on the $\frac{1}{2}-$axis($z\in\Gamma_c$) correspond to skinny-rhombic lattices. A detailed description of the strain tensors under the group $\hbox{GL}(2,\mathbb{Z})$ is depicted in \cite{Conti2004A}.
In particular, the intersection point $\{i\}=\Gamma_a\cap\Gamma_b$ represents the square lattice, while the intersection point $\{e^{i\frac{\pi}{3}}\}=\Gamma_b\cap\Gamma_c$ represents the hexagonal lattice. The points $i$ and $e^{i\frac{\pi}{3}}$(hence square and hexagonal lattice) play a central role in our analysis. Another viewpoint is that
the strain tensors $\tilde{\mathbf{C}}$(hence $z$) represent the local deformation and function as the role of phase in phase transition.

Let $L$ be a two dimensional lattice spanned by two linearly independent vectors $\mathbf{w}_1, \mathbf{w}_2$. In the complex
coordinates, the shape of the lattice $L$ is completely determined by a new complex parameter
$$z=\frac{\mathbf{w}_2}{\mathbf{w}_1}, \;\hbox{or}\;\;\mathbf{w}_2=z\mathbf{w}_1.
$$
In this way, a Bravais lattice with unit density can be parameterized by $L =\sqrt{\frac{1}{\Im(z)}}\Big({\mathbb Z}\oplus z{\mathbb Z}\Big)$   where $z \in \mathbb{H}$.

A classical way to construct a strain energy density satisfying \eqref{V3} is by the Klein modular function $J$ or its variants(\cite{Conti2019,Parry1998}), while a more general way is by the summation on the lattices through a external potential $\mathbf{v}$, namely by
\begin{equation}\aligned\label{V5}f(z):=\sum_{\mathbb{P}\in L}\mathbf{v}(|\mathbb{P}|^2).
\endaligned\end{equation}
It is straightforward to check that $f(z)$ satisfies \eqref{V3}. Note that Folkins \cite{Folkins1991} proposed the lattice summation \eqref{V5}
to describe the energy per atom on Bravais lattices, harmonic analysis of functionals \eqref{V5} are given there.
A function of the form in \eqref{V5}(modular invariant)
appears in number theory and gives rise to natural candidates for strain energy functions with symmetry group.
Along this line, one obtains that
\begin{reformulation}\label{R1}
Landau's theory of plannar crystal plasticity can be formulated to a variational problem
\begin{equation}\aligned\label{V6}\min_{L, |L|=\rho>0}\sum_{\mathbb{P}\in L}\mathbf{v}(\alpha\cdot|\mathbb{P}|^2),
\endaligned\end{equation}
here $\mathbf{v}$ is the background function related to the plasticity$($involving parameters$)$, and $\rho$ is the density of the crystal$($lattice$)$. The parameter $\alpha$ plays the role of "temperature". The summation represents free energy density of the crystals at a fixed "temperature" $\alpha$.
The shape of $L$ determines the phase of the crystals.
\end{reformulation}
This reformulation consisting in lattice summation form looks new and different, comparing the classical form of phase transitions consisting in integral form(see e.g. \cite{Conti2007,Muller} and the references therein).

An interesting open question closely related to the Reformulation is

\begin{problemA}[Square-hexagonal phase transition, \cite{Conti2004A,Conti2004B}, \cite{Conti2004A} Section 4, page 82]\label{ProA} Obtain a model energy for the square-hexagonal transformation, i.e., find a strain energy function$($satisfying \eqref{V3}$)$ such that the global minimum is always either the square or the hexagonal state when changing a parameter, the system goes through the following three regimes:
\begin{itemize}
  \item $(i)$ the square state is the global minimum and the hexagonal is unstable(precisely, a saddle point);
  \item $(ii)$ both the square and the hexagonal states are local minimum;
  \item $(iii)$ the hexagonal is the global minimum, and the square is unstable.
\end{itemize}

\end{problemA}
For phase transitions in crystals, Landau 1936(\cite{Landau1936}) proposed a theory to explain it, namely, finding the minimum of the density function, further
this density function should has some symmetries. This theory was further developed by many mathematicians(\cite{Ericksen1980,Folkins1991,Conti2004A,Conti2004B,Parry1998} etc.). In two dimensional crystals, these symmetries was further identified by $\hbox{SL}(2,\mathbb{Z})$.

A mathematical theory to explain hexagonal to square lattice phase transitions is to find a strain energy function with a parameter such that
its minimum is achieved at hexagonal lattice or square lattice as the parameter changing, without passing through the rhombic lattice.
This problem is initiated by Landau in 1936 and justified by many mathematicians and physicists.

Complex six-order polynomials were constructed to answer Problem A under the assumption that the reference lattice is either square or hexagonal(\cite{Conti2004A}), and recently in \cite{Conti2019,Z2023,Z2020}, a modular invariant function
\begin{equation}\aligned\nonumber
|J(z)-1|+\alpha |J(z)|^{2/3},
\endaligned\end{equation}
was constructed to answer Problem A by numerical computation restricted on the unit arc, where $J$ is the Klein invariant function.  These two results are interesting progress to Problem A, while a complete and rigorous proof  without assumption on the reference lattice or restriction on the arc are needed.  Our result provides  a simple but rigorously justified  strain energy function (satisfying \eqref{V1}) to the mathematical theory of Landau's phase transition, by invoking the form in \eqref{V5}.

Our idea is motivated from classical number theory, this is partially inspired by Parry \cite{Parry1998}(page 2), where he remarked that"{\it My purpose, here, is to describe and elaborate upon methods from classical complex analysis and number theory which have a natural application in this area$($continuum mechanics of phase transitions in crystals$)$}". In 1988, number theorist Montgomery \cite{Mon1988} proved the following celebrated result:
\begin{thmA}[Hexagonal lattice] \label{Theorem4} For all $\alpha> 0$,
\begin{equation}\aligned
\min_{L}\sum_{\mathbb{P}\in L, |L|=1} e^{- \pi\alpha|\mathbb{P}|^2}\;\;\hbox{is achieved at hexagonal lattice}.\;\;
\endaligned\end{equation}

\end{thmA}
Here we fix the volume/density of the lattice to be $1$, the hexagonal lattice is $L =\sqrt{\frac{2}{\sqrt3}}\Big({\mathbb Z}\oplus e^{i\frac{\pi}{3}}{\mathbb Z}\Big)$.
Montgomery's Theorem has profound application in mathematical physics(see e.g. \cite{Serfaty2012}).
Perhaps, the modular invariant function most related to $\sum_{\mathbb{P}\in L, |L|=1} e^{- \pi\alpha|\mathbb{P}|^2}$ is $\sum_{\mathbb{P}\in L, |L|=1} |\mathbb{P}|^2e^{- \pi\alpha|\mathbb{P}|^2}$ which is the derivative of Montgomery functional with respect to the parameter $\alpha$. In contrast to Montgomery's Theorem A, we establish that

\begin{theorem}[Hexagonal-Square-Rectangular Phase Transitions]\label{Th1}Assume that $\alpha>0$, then there exist $\alpha_a<\alpha_b\in(\frac{4}{5},1)$ such that
\begin{equation}\aligned\nonumber
\min_{L}\sum_{\mathbb{P}\in L, |L|=1}|\mathbb{P}|^2 e^{- \pi\alpha|\mathbb{P}|^2}\;\;\hbox{is achieved at}\;\;
\begin{cases}
\;\;rectangular\; lattice, &\hbox{if}\;\; \alpha\in(0,\alpha_a),\\
\;\;square\; lattice\; &\hbox{if}\;\; \alpha\in[\alpha_a,\alpha_b),\\
\;\;square\;\hbox{or}\;hexagonal\; lattice &\hbox{if}\;\; \alpha=\alpha_b,\\
\;\;hexagonal\;lattice, &\hbox{if}\;\; \alpha\in(\alpha_b,\infty).
\end{cases}
\endaligned\end{equation}
Here, numerically,
\begin{equation}\aligned\nonumber
\alpha_a:=0.8947042694\cdots,\;\;\alpha_b:=0.9203340927\cdots.
\endaligned\end{equation}

\end{theorem}

A square lattice can be parameterized by $L =\Big({\mathbb Z}\oplus i{\mathbb Z}\Big)$, and a rectangular lattice is
$L =\sqrt{\frac{1 }{y}}\Big({\mathbb Z}\oplus iy{\mathbb Z}\Big)$ for $y>1$. Theorem \ref{Th1} has independent interest from number theory along the line of Theorem A, and has many consequences in phase transition and particle physics.

Problem A is answered by the following result (after reordering and re-scaling the parameter, e.g. $\alpha\mapsto\frac{1}{\alpha}, \alpha\mapsto k\cdot\alpha$), followed directly by Theorem \ref{Th1}. Note that the parameter $\alpha$ or $\frac{1}{\alpha}$ now plays the role of "temperature" in the phase transitions, the thresholds $\alpha_a, \alpha_b$ play the similar role of temperatures $0^o,100^o$ in solid-liquid, and liquid-gas phase transitions of water.
\begin{corollary}
[Hexagonal-Square Phase Transitions]\label{Coro1}Assume that $\alpha\geq\alpha_a$, then
\begin{equation}\aligned\nonumber
\min_{L}\sum_{\mathbb{P}\in L, |L|=1}|\mathbb{P}|^2 e^{- \pi\alpha|\mathbb{P}|^2}\;\;\hbox{is achieved at}\;\;
\begin{cases}
\;\;square\; lattice\; &\hbox{if}\;\; \alpha\in[\alpha_a,\alpha_b),\\
\;\;square\;\hbox{or}\;hexagonal\; lattice &\hbox{if}\;\; \alpha=\alpha_b,\\
\;\;hexagonal\;lattice, &\hbox{if}\;\; \alpha\in(\alpha_b,\infty).
\end{cases}
\endaligned\end{equation}
Here, numerically,
\begin{equation}\aligned\nonumber
\alpha_a:=0.8947042694\cdots,\;\;\alpha_b:=0.9203340927\cdots.
\endaligned\end{equation}

\end{corollary}
\begin{remark} By Corollary \ref{Coro1}, a strain energy function satisfying \eqref{V1} with parameter $\alpha\geq \alpha_a$ is
\begin{equation}\aligned\label{PPP}
\sum_{(m,n)\in\mathbb{Z}^2}
e^{-\pi\alpha \frac{\mathbf{C}_{11} }{\sqrt{\mathbf{C}_{11}\mathbf{C}_{22}-\mathbf{C}_{12}^2} }|m\frac{\mathbf{C}_{12}+\sqrt{\mathbf{C}_{11}C_{22}-\mathbf{C}_{12}^2}\cdot i }{\mathbf{C}_{11}}+n|^2},
\endaligned\end{equation}
which admits the hexagonal-square phase transitions,
here $\mathbf{C}=(\mathbf{C}_{ij})$ is the Cauchy-Green tensor.
\end{remark}

Via the classification and representation of 2d lattices in \cite{Conti2004A}, we have
\begin{corollary}
[Minimization on the Cauchy-Green tensor]\label{Coro2}Assume that $\alpha\geq\alpha_a$, then
\begin{equation}\aligned\nonumber
&\min_{\mathbf{C}>0,\mathbf{C}^T=\mathbf{C} }\sum_{(m,n)\in\mathbb{Z}^2}
e^{-\pi\alpha \frac{\mathbf{C}_{11} }{\sqrt{\mathbf{C}_{11}\mathbf{C}_{22}-\mathbf{C}_{12}^2} }|m\frac{\mathbf{C}_{12}+\sqrt{\mathbf{C}_{11}C_{22}-\mathbf{C}_{12}^2}\cdot i }{\mathbf{C}_{11}}+n|^2}\;\;\\
\hbox{is achieved at}\;\;
&\begin{cases}
\;\;\mathbf{C}_{11}=\mathbf{C}_{22}=2\mathbf{C}_{12}>0, &\hbox{if}\;\; \alpha\in[\alpha_a,\alpha_b],\\
\;\;\mathbf{C}_{11}=\mathbf{C}_{22}>0, \mathbf{C}_{12}=0, &\hbox{if}\;\; \alpha\in[\alpha_b,\infty).
\end{cases}
\endaligned\end{equation}
Here, numerically,
\begin{equation}\aligned\nonumber
\alpha_a:=0.8947042694\cdots,\;\;\alpha_b:=0.9203340927\cdots.
\endaligned\end{equation}

\end{corollary}

By Corollary \ref{Coro1}, the strain energy function to {\bf Problem A} is $\sum_{\mathbb{P}\in L, |L|=1}|\mathbb{P}|^2 e^{- \pi\alpha|\mathbb{P}|^2}$ for $\alpha\geq\alpha_a$ (see also \eqref{PPP}). This provides an affirmative and positive answer to {\bf Problem A}.

By the parametrization of the Bravais lattice with fixed density $\rho$($\rho>0$), one has
\begin{equation}\aligned\label{P1}
\sum_{\mathbb{P}\in L, |L|=\rho}|\mathbb{P}|^2 e^{- \pi\alpha|\mathbb{P}|^2}=\sum_{(m,n)\in\mathbb{Z}^2 }
\frac{|mz+n|^2}{\Im(z) }
 e^{- \frac{\pi\cdot\alpha\cdot \rho |mz+n|^2}{\Im(z) }}.
\endaligned\end{equation}
By the parametrization, with loss of generality, we assume that density the lattice is 1 in statement of previous theorems.
The proof of Theorem \ref{Th1} is based on
\begin{theorem}[Minimization]     \label{Th2}
For all $\alpha>0$,
\begin{equation}\aligned
\min_{z\in\mathbb{H}}\sum_{(m,n)\in\mathbb{Z}^2 }
\frac{|mz+n|^2}{\Im(z) }
 e^{- \frac{\pi\alpha |mz+n|^2}{\Im(z) }}\;\;\hbox{is achieved at}\;\;
\begin{cases}
\;\;iy_\alpha, &\hbox{if}\;\; \alpha\in(0,\alpha_a),\\
\;\;i\; &\hbox{if}\;\; \alpha\in[\alpha_a,\alpha_b),\\
\;\;i\;\hbox{or}\;e^{i\frac{\pi}{3}} &\hbox{if}\;\; \alpha=\alpha_b,\\
\;\;e^{i\frac{\pi}{3}}, &\hbox{if}\;\; \alpha\in(\alpha_b,\infty).
\end{cases}
\endaligned\end{equation}
Here, numerically,
\begin{equation}\aligned\nonumber
\alpha_a:=0.8947042694\cdots,\;\;\alpha_b:=0.9203340927\cdots,
\endaligned\end{equation}

and $y_\alpha>1$, further $y_\alpha$ is increasing as $\alpha$ decreasing, namely,
 $$\frac{d y_\alpha }{d\alpha}<0.$$
 Let
\begin{equation}\aligned
\nonumber
{M} (\alpha, z)=\sum_{(m,n)\in\mathbb{Z}^2} \frac{|mz+n|^2 }{\Im(z) } e^{- \pi\alpha \frac{|mz+n|^2}{\Im(z) }},
\;\;\theta (\alpha, z)=\sum_{(m,n)\in\mathbb{Z}^2}e^{- \pi\alpha \frac{|mz+n|^2}{\Im(z) }}
\endaligned\end{equation}
then analytically,
$\alpha_b$ is the unique solution of $$
M(\alpha,i)=M(\alpha,e^{i\frac{\pi}{3}}),\;\;\hbox{for}\;\;\alpha\in[\frac{5}{6},1],$$
$\frac{1}{\alpha_a}$ is the unique solution of
\begin{equation}\aligned\nonumber
\theta_{yy}(\alpha,i)=\pi\alpha M_{yy}(\alpha,i),\;\;\hbox{for}\;\;\alpha\in[1,\frac{9}{8}].
\endaligned\end{equation}

\end{theorem}
The $y_\alpha$ in Theorem \ref{Th2} can be located precisely for each $\alpha$.

{\bf
\begin{table}[!htbp]\label{TableB}
\caption{{\bf Phase transitions in crystals: the variational solutions.}}
\label{VP}
\centering
\begin{tabular}{|c|c|c|c|}

\hline

The ranges of $\alpha$  & $(0,\alpha_a)$ & $[\alpha_a,\alpha_b]$ & $(\alpha_b,\infty)$ \\

\hline

Minimizers &  $iy_\alpha$ & $i$ & $e^{i\frac{\pi}{3}}$  \\

\hline
Optimal lattice shapes & Rectangular&  Square & Hexagonal   \\
\hline

\end{tabular}
\end{table}
}
An illustration of Theorem \ref{Th1} is displayed in Table \ref{TableB}.
Theorem \ref{Th1} has a direct application in phase transitions in crystals.

\begin{remark} As the density $\alpha$ decreasing from $\infty$ to $0$, one sees the {\bf phase transition of hexagonal-square-rectangular lattice shapes},
there is {\bf no rhombic phase} appearing.

\end{remark}

\begin{remark}
This is the first rigorous proof of hexagonal-square-rectangular phase transition.
\end{remark}

\begin{remark} Part of item $(4)$, i.e., the case $\alpha\geq1$ was proved in our previous work$($\cite{LuoA}$)$.

\end{remark}

We provide more examples to admit the hexagonal to square phase transitions(answering {\bf Problem A}).
\begin{theorem}[Hexagonal-Square Phase Transitions]\label{Th3}Assume that $\alpha>0$ and $\gamma>0$ then there exists $\alpha_{\gamma_1}<\alpha_{\gamma_2}\in(0,1)$ such that
\begin{equation}\aligned\nonumber
\min_{L}\sum_{\mathbb{P}\in L, |L|=1}(|\mathbb{P}|^2+\gamma) e^{- \pi\alpha|\mathbb{P}|^2}\;\;\hbox{is achieved at}\;\;
\begin{cases}
\;\;square\; lattice\; &\hbox{if}\;\; \alpha\in[\alpha_{\gamma_1},\alpha_{\gamma_2}),\\
\;\;square\;\hbox{or}\;hexagonal\; lattice &\hbox{if}\;\; \alpha=\alpha_{\gamma_2},\\
\;\;hexagonal\;lattice, &\hbox{if}\;\; \alpha\in(\alpha_{\gamma_2},\infty).
\end{cases}
\endaligned\end{equation}

\end{theorem}

Some comments are in order. We obtain several classes of modular invariant functions admitting
hexagonal to square phase transitions. This is the first time to obtain such functions, although it was
conjectured to exist for long time.  On the other hand, continuous hexagonal-rhombic-square-rectangular phase transitions have been observed and rigorously established in many modular invariant functions.  See \cite{Luo2022, Luo2023}.

Given by Theorems \ref{Th1}-\ref{Th3}, one may incline to guess that such results hold for
$$\sum_{\mathbb{P}\in L, |L|=1}|\mathbb{P}|^4 e^{- \pi\alpha|\mathbb{P}|^2},$$ however this is not true (by numerically computations).

Hexagonal and square lattices have the dominated role in two dimensional Bravais lattices (as seen in Theorem \ref{Th1}). While in dimension 3, the FCC and BCC lattices play a similar key role. In contrast to Theorem \ref{Th1}, we propose an open problem:
\begin{open problem}\label{PPA}
Classify
\begin{equation}\aligned\nonumber
\min_{L}\sum_{\mathbb{P}\in L, |L|=1}|\mathbb{P}|^2 e^{- \pi\alpha|\mathbb{P}|^2},\;\;\hbox{here}\;\;L\;\;\hbox{is a 3-dimensional lattice and}\;\;\alpha>0.
\endaligned\end{equation}
\end{open problem}

The paper is organized as follows: in Section 2, we collect some basic properties of the functionals and some basic estimates of derivatives of Jacobi theta functions.

In Section 3, we prove a transversal monotonicity of the functionals, as a consequence, we show that for certain range of $\alpha$, the minimizers of the functionals must locate on the partial boundary of the half fundamental domain, namely, $\Gamma_a\cup \Gamma_b$(see Picture \ref{LJFFF}).

In Section 4, we establish a minimum principle of modular invariant functions and deduce several second-order estimates in using the minimum principle, as a consequence, we prove that for certain range of $\alpha$, the the minimizers of the functionals must locate on the partial boundary of the half fundamental domain, namely, $\Gamma_b$(see Picture \ref{LJFFF}).

 By the main results in Sections 3 and 4, we obtain that for the full range of $\alpha(\alpha>0)$,
the minimizers of the functionals must locate on $\Gamma_a\cup \Gamma_b$(see Picture \ref{LJFFF}).

In Sections 5 and 6, we develop effective methods to analyze the functionals on the vertical line $\Gamma_a$ and the $\frac{1}{4}-$ arc $\Gamma_b$(see Picture \ref{LJFFF}) respectively. Our methods in Sections 3-6 are powerful and can be applied to many other related problems.

 Finally, we give the proof of our main Theorem in Section 7.

\section{Preliminaries }
\setcounter{equation}{0}

In this section, we collect  some  simple symmetry properties  of the functionals and the associated fundamental domain, and also the estimates of derivatives Jacobi theta functions to be used in later sections.

Let
$
\mathbb{H}
$
 denote the upper half plane and  $\mathcal{S} $ denote the modular group
\begin{equation}\aligned\label{modular}
\mathcal{S}:=\hbox{SL}_2(\mathbb{Z})=\{
\left(
  \begin{array}{cc}
    a & b \\
    c & d \\
  \end{array}
\right), ad-bc=1, a, b, c, d\in\mathbb{Z}
\}.
\endaligned\end{equation}

We use the following definition of fundamental domain which is slightly different from the classical definition (see \cite{Mon1988}):
\begin{definition} [page 108, \cite{Eva1973}]
The fundamental domain associated to group $G$ is a connected domain $\mathcal{D}$ satisfies
\begin{itemize}
  \item For any $z\in\mathbb{H}$, there exists an element $\pi\in G$ such that $\pi(z)\in\overline{\mathcal{D}}$;
  \item Suppose $z_1,z_2\in\mathcal{D}$ and $\pi(z_1)=z_2$ for some $\pi\in G$, then $z_1=z_2$ and $\pi=\pm Id$.
\end{itemize}
\end{definition}

By Definition 2.1, the fundamental domain associated to modular group $\mathcal{S}$ is
\begin{equation}\aligned\label{Fd1}
\mathcal{D}_{\mathcal{S}}:=\{
z\in\mathbb{H}: |z|>1,\; -\frac{1}{2}<x<\frac{1}{2}
\}
\endaligned\end{equation}
which is open.  Note that the fundamental domain can be open. (See [page 30, \cite{Apo1976}].)

Next we introduce another group related  to the functionals $\theta(\alpha;z)$. The generators of the group are given by
\begin{equation}\aligned\label{GroupG1}
\mathcal{G}: \hbox{the group generated by} \;\;\tau\mapsto -\frac{1}{\tau},\;\; \tau\mapsto \tau+1,\;\;\tau\mapsto -\overline{\tau}.
\endaligned\end{equation}

It is easy to see that
the fundamental domain associated to group $\mathcal{G}$ denoted by $\mathcal{D}_{\mathcal{G}}$ is
\begin{equation}\aligned\label{Fd3}
\mathcal{D}_{\mathcal{G}}:=\{
z\in\mathbb{H}: |z|>1,\; 0<x<\frac{1}{2}
\}.
\endaligned\end{equation}

The following lemma characterizes the fundamental symmetries of the theta functions $\theta (s; z)$. The proof is easy so we omit it.
Let
\begin{equation}\aligned
\nonumber
{M} (\alpha, z):&=\sum_{\mathbb{P}\in L,\; |L|=1} |\mathbb{P}|^2e^{- \pi\alpha |\mathbb{P}|^2}=\sum_{(m,n)\in\mathbb{Z}^2} \frac{|mz+n|^2 }{\Im(z) } e^{- \pi\alpha \frac{|mz+n|^2}{\Im(z) }}.
\endaligned\end{equation}
\begin{lemma}\label{G111} For any $\alpha>0$, any $\gamma\in \mathcal{G}$ and $z\in\mathbb{H}$,
$M(\alpha; \gamma(z))=M(\alpha;z)$.
\end{lemma}

Next we need  some delicate analysis of the Jacobi theta function which is defined as

\begin{equation}\aligned\nonumber
\vartheta_J(z;\tau):=\sum_{n=-\infty}^\infty e^{i\pi n^2 \tau+2\pi i n z}.
 \endaligned\end{equation}
 The classical one-dimensional theta function  is given by
\begin{equation}\aligned\label{TXY}
\vartheta(X;Y):=\vartheta_J(Y;iX)=\sum_{n=-\infty}^\infty e^{-\pi n^2 X} e^{2n\pi i Y}.
 \endaligned\end{equation}
By the Poisson summation formula, it holds that
\begin{equation}\aligned\label{PXY}
\vartheta(X;Y)=X^{-\frac{1}{2}}\sum_{n=-\infty}^\infty e^{-\frac{\pi(n-Y)^2}{X}}.
 \endaligned\end{equation}
To estimate bounds of quotients of derivatives of $\vartheta(X:Y)$, we denote that

\begin{equation}\aligned\label{mmmx}
\mu(X):=\sum_{n=2}^\infty n^2 e^{-\pi(n^2-1)X},\;\;
\nu(X):=\sum_{n=2}^\infty n^4 e^{-\pi(n^2-1)X}.
\endaligned\end{equation}

The following three lemmas are proved in \cite{LuoA}.
\begin{lemma}\cite{LuoA}\label{Lemma2a} Assume that $Y>0, k\in \mathbb{N}^+$. It holds that
\begin{itemize}

  \item $(1):$
$
|\frac{\vartheta_Y(X;k Y)}{\vartheta_Y(X;Y)}|\leq k\cdot\frac{1+\mu(X)}{1-\mu(X)}\;\;\hbox{for}\;\;X>\frac{1}{5};
$
  \item $(2):$
$
|\frac{\vartheta_Y(X;k Y)}{\vartheta_Y(X;Y)}|\leq k\cdot\frac{1}{\pi}e^{\frac{\pi}{4X}}\;\;\hbox{for}\;\;X<\frac{\pi}{\pi+2}.
$
\end{itemize}

\end{lemma}

To give the desired estimates, we further need the following
\begin{lemma}\cite{LuoA}\label{Lemma2b}Assume that $Y>0, k\in \mathbb{N}^+$. It holds that
\begin{itemize}
  \item $(1):$
$
|\frac{\vartheta_{XY}(X;k Y)}{\vartheta_{XY}(X;Y)}|\leq k\cdot\frac{1+\nu(X)}{1-\nu(X)}\;\;\hbox{for}\;\;X\geq\frac{3}{10};
$
  \item $(2):$
$
|\frac{\vartheta_{XY}(X;k Y)}{\vartheta_{Y}(X;Y)}|\leq k\pi\cdot \frac{1+\nu(X)}{1-\mu(X)}\;\;\hbox{for}\;\;X\geq\frac{1}{5}.
$
 \item
$(3):$ and for $k=1$, we have the more precise bound
$
|\frac{\vartheta_{XY}(X;Y)}{\vartheta_{Y}(X;Y)}|\leq\pi\cdot \frac{1+\nu(X)}{1+\mu(X)}
$ for $X\geq\frac{1}{5}$.
\end{itemize}

\end{lemma}

We shall establish the following estimates which are useful in the next section.
\begin{lemma}\cite{LuoA}\label{Lemma2c} For $X\leq\frac{1}{2}$ and any $Y>0$, $k\in \mathbb{N}^+$, it holds that
\begin{itemize}
  \item $(1):$ $|\frac{\vartheta_{XY}(X;Y)}{\vartheta_Y(X;Y)}|\leq \frac{3}{2}X^{-1}(1+\frac{\pi}{6}\frac{1}{X})$;
  \item $(2):$ $|\frac{\vartheta_{XY}(X;k Y)}{\vartheta_Y(X;Y)}|\leq \frac{3k}{2\pi}X^{-1}(1+\frac{\pi}{6}\frac{1}{X})e^{\frac{\pi}{4X}}$.
\end{itemize}

\end{lemma}

\section{The transversal monotonicity}

Recall that the lattice with unit density is parameterized by $L =\sqrt{\frac{1}{\Im(z)}}\Big({\mathbb Z}\oplus z{\mathbb Z}\Big)$.
The classical Theta function is defined as
\begin{equation}\aligned
 \label{T}
\theta (\alpha, z):&=\sum_{\mathbb{P}\in L,\; |L|=1} e^{- \pi\alpha |\mathbb{P}|^2}=\sum_{(m,n)\in\mathbb{Z}^2} e^{-\pi\alpha \frac{1 }{\Im(z) }|mz+n|^2}.
\endaligned\end{equation}
We then define the ${M}$ function as
\begin{equation}\aligned
 \label{J}
{M} (\alpha, z):&=\sum_{\mathbb{P}\in L,\; |L|=1} |\mathbb{P}|^2e^{- \pi\alpha |\mathbb{P}|^2}=\sum_{(m,n)\in\mathbb{Z}^2} \frac{|mz+n|^2 }{\Im(z) } e^{- \pi\alpha \frac{|mz+n|^2}{\Im(z) }}.
\endaligned\end{equation}

We also recall that right-half fundamental domain
\begin{equation}\aligned\nonumber
\mathcal{D}_{\mathcal{G}}:=\{
z\in\mathbb{H}: |z|>1,\; 0<x<\frac{1}{2}
\}.
\endaligned\end{equation}
The boundaries of $ \mathcal{D}_{\mathcal{G}}$ are divided into two pieces
\begin{equation}\aligned\nonumber
\Gamma_a:&=\{
z\in\mathbb{H}: z=iy,\; y\geq1
\},\\
\Gamma_b:&=\{
z\in\mathbb{H}: z=e^{i\theta},\; \theta\in[\frac{\pi}{3},\frac{\pi}{2}]
\},
\endaligned\end{equation}
\begin{equation}\aligned\nonumber
\Gamma=\Gamma_a\cup\Gamma_b.
\endaligned\end{equation}
In \cite{LuoA}, we already prove that
\begin{equation}\aligned\nonumber
\min_{z\in\mathbb{H}}M(\alpha,z)\;\;\hbox{is achieved at}\;\; e^{i\frac{\pi}{3}}\;\;\hbox{for}\;\;\alpha\geq1.
\endaligned\end{equation}

In this Section, we aim to prove that
\begin{proposition}\label{Prop1} For $\alpha\in(0,0.9155730607)$,
\begin{equation}\aligned\nonumber
\min_{z\in\mathbb{H}}M(\alpha,z)=\min_{z\in\overline{\mathcal{D}_{\mathcal{G}}}}M(\alpha,z)=\min_{z\in\Gamma}M(\alpha,z).
\endaligned\end{equation}
\end{proposition}
The proof of Proposition \ref{Prop1} is based on the following
\begin{proposition}\label{Prop2} For $\alpha\in(0,0.9155730607)$,
\begin{equation}\aligned\nonumber
\frac{\partial}{\partial x}M(\alpha,z)>0,\;\;\hbox{for}\;\;z\in\mathcal{D}_{\mathcal{G}}.
\endaligned\end{equation}
\end{proposition}

We first state a useful lemma.

\begin{lemma}[Duality of $M$ evaluating at $\frac{1}{\alpha}$ and $\alpha$]\label{Lemma1} Assume that $\alpha>0$, then
\begin{equation}\aligned\nonumber
M(\frac{1}{\alpha},z)=\frac{\alpha^2}{\pi}\Big(
\theta(\alpha,z)-\pi\alpha M(\alpha,z)
\Big).
\endaligned\end{equation}

\end{lemma}
\begin{proof} By Fourier transform,
\begin{equation}\aligned\label{Ta}
\theta(\frac{1}{\alpha},z)=\alpha\theta(\alpha,z).
\endaligned\end{equation}
Taking derivative with respect to $\alpha$ on \eqref{Ta}, one has
\begin{equation}\aligned\label{Eq1}
-\frac{1}{\alpha^2}\frac{\partial}{\partial \alpha}\theta(\frac{1}{\alpha},z)
=\theta(\alpha,z)+\alpha\frac{\partial}{\partial \alpha}\theta(\alpha,z).
\endaligned\end{equation}

The result then follows by \eqref{Eq1} and Lemma \ref{Lemma2}.
\end{proof}

The new function $M$ has a close relation the classical Theta function.

\begin{lemma}[A relation between the functions $M$ and $\theta$]\label{Lemma2}
\begin{equation}\aligned\nonumber
M(\alpha,z)=-\frac{1}{\pi}\frac{\partial}{\partial \alpha}\theta(\alpha,z).
\endaligned\end{equation}

\end{lemma}
\begin{proof}
It is based an cute observation from \eqref{T} and \eqref{J}.

\end{proof}

By Lemma \ref{Lemma1}, Proposition \ref{Prop2} is equivalent to
\begin{proposition}\label{Prop3} For $\alpha\in(1.092212127,\infty)$,
\begin{equation}\aligned\nonumber
\frac{\partial}{\partial x}\Big(
\theta(\alpha,z)-\pi\alpha M(\alpha,z)
\Big)>0,\;\;\hbox{for}\;\;z\in\mathcal{D}_{\mathcal{G}}.
\endaligned\end{equation}
\end{proposition}

In the rest of this section, we give the proof of Proposition \ref{Prop3}. To prove Proposition \ref{Prop3},
we give some auxiliary lemmas firstly.

By the explicit expression of $\theta$ in \eqref{T}, one can express $\theta$ in terms of one dimensional theta function defined in
\eqref{TXY} or \eqref{PXY}.
\begin{lemma}[Reduction of dimension]\label{Lemma3} It holds that
\begin{equation}\aligned\nonumber
\theta(\alpha,z)
&=\sqrt{\frac{y}{\alpha}}\sum_{n\in\mathbb{Z}} e^{-\pi\alpha y n^2}\vartheta(\frac{y}{\alpha};nx)\\
\frac{\partial}{\partial x}\theta(\alpha,z)
&=2\sqrt{\frac{y}{\alpha}}\sum_{n=1}^\infty n e^{-\pi\alpha y n^2}\vartheta_Y(\frac{y}{\alpha};nx)
.
\endaligned\end{equation}

\end{lemma}

By Lemmas \ref{Lemma2} and \ref{Lemma3}, one has
\begin{lemma}[An exponentially decaying expansion of $\frac{\partial}{\partial x}M(\alpha,z)$]\label{Lemma4}
\begin{equation}\aligned\nonumber
\frac{\partial}{\partial x}M(\alpha,z)&=\frac{1}{\pi}\Big(
\sqrt{y}\alpha^{-\frac{3}{2}}\sum_{n=1}^\infty n e^{-\pi\alpha y n^2}\vartheta_Y(\frac{y}{\alpha};nx)
+2\pi y^{\frac{3}{2}}\alpha^{-\frac{1}{2}}\sum_{n=1}^\infty n^3 e^{-\pi\alpha y n^2}\vartheta_Y(\frac{y}{\alpha};nx)\\
&\;\;\;\;+2 y^{\frac{3}{2}}\alpha^{-\frac{5}{2}}\sum_{n=1}^\infty n e^{-\pi\alpha y n^2}\vartheta_{XY}(\frac{y}{\alpha};nx)\Big).
\endaligned\end{equation}

\end{lemma}

By Lemmas \ref{Lemma3} and \ref{Lemma4}, we get
\begin{lemma}[An exponentially decaying expansion of $\frac{\partial}{\partial x}\Big(
\theta(\alpha,z)-\pi\alpha M(\alpha,z)
\Big)$]\label{Lemma5}
\begin{equation}\aligned\nonumber
\frac{\partial}{\partial x}\Big(
\theta(\alpha,z)-\pi\alpha M(\alpha,z)
\Big)
=&\sqrt{y}\alpha^{-\frac{3}{2}}\cdot(-\vartheta_Y(\frac{y}{\alpha};x))\cdot e^{-\pi\alpha y}
\Big(
2\pi y\alpha^2 \sum_{n=1}^\infty n^3 e^{-\pi\alpha y (n^2-1)}\frac{\vartheta_Y(\frac{y}{\alpha};nx)}{\vartheta_Y(\frac{y}{\alpha};x)}\\
&-2y\sum_{n=1}^\infty n e^{-\pi\alpha y (n^2-1)}\frac{\vartheta_{XY}(\frac{y}{\alpha};nx)}{-\vartheta_Y(\frac{y}{\alpha};x)}
-\alpha\sum_{n=1}^\infty n e^{-\pi\alpha y (n^2-1)}\frac{\vartheta_Y(\frac{y}{\alpha};nx)}{\vartheta_Y(\frac{y}{\alpha};x)}
\Big)
\endaligned\end{equation}
\end{lemma}

For convenience, one denotes that
\begin{equation}\aligned\nonumber
L(\alpha,x,y):=
&
2\pi y\alpha^2 \sum_{n=1}^\infty n^3 e^{-\pi\alpha y (n^2-1)}\frac{\vartheta_Y(\frac{y}{\alpha};nx)}{\vartheta_Y(\frac{y}{\alpha};x)}
-2y\sum_{n=1}^\infty n e^{-\pi\alpha y (n^2-1)}\frac{\vartheta_{XY}(\frac{y}{\alpha};nx)}{-\vartheta_Y(\frac{y}{\alpha};x)}\\
&-\alpha\sum_{n=1}^\infty n e^{-\pi\alpha y (n^2-1)}\frac{\vartheta_Y(\frac{y}{\alpha};nx)}{\vartheta_Y(\frac{y}{\alpha};x)}
.
\endaligned\end{equation}
Then Lemma \ref{Lemma5} can be rewritten as
\begin{equation}\aligned\label{Jx}
\frac{\partial}{\partial x}\Big(
\theta(\alpha,z)-\pi\alpha M(\alpha,z)
\Big)
=&\sqrt{y}\alpha^{-\frac{3}{2}}\cdot(-\vartheta_Y(\frac{y}{\alpha};x))\cdot e^{-\pi\alpha y}
\cdot L(\alpha,x,y).
\endaligned\end{equation}
To obtain a qualitative result from \eqref{Jx}, one notes that
\begin{lemma}[\cite{Luo2022}]\label{Lemma6} For $y>0,\alpha>0$ and $x\in[0,\frac{1}{2}]$,
\begin{equation}\aligned\nonumber
-\vartheta_Y(\frac{y}{\alpha};x)\geq0.
\endaligned\end{equation}
\end{lemma}
The function $L(\alpha,x,y)$ can be estimated by
\begin{lemma}\label{Lemma7} For $y\geq\frac{\sqrt3}{2},\alpha\geq1$ and $x\in\R$,
\begin{equation}\aligned\nonumber
L(\alpha,x,y)=2\pi y\alpha^2-\alpha-2y\frac{\vartheta_{XY}(\frac{y}{\alpha};x)}{-\vartheta_Y(\frac{y}{\alpha};x)}
+O(e^{-3\pi \alpha y}).
\endaligned\end{equation}
In fact,
\begin{equation}\aligned\nonumber
L(\alpha,x,y)&=2\pi y\alpha^2-\alpha-2y\frac{\vartheta_{XY}(\frac{y}{\alpha};x)}{-\vartheta_Y(\frac{y}{\alpha};x)}\\
&+
\sum_{n=2}^\infty e^{-\pi\alpha y(n^2-1)}\big(2\pi y\alpha^2 n^3\frac{\vartheta_{Y}(\frac{y}{\alpha};nx)}{\vartheta_Y(\frac{y}{\alpha};x)}-
2yn\frac{\vartheta_{XY}(\frac{y}{\alpha};nx)}{-\vartheta_Y(\frac{y}{\alpha};x)}-\alpha n\frac{\vartheta_{Y}(\frac{y}{\alpha};nx)}{\vartheta_Y(\frac{y}{\alpha};x)}\big).
\endaligned\end{equation}
\end{lemma}

We shall further investigate the lower bound of the major term in $L(\alpha,x,y)$ from Lemma \ref{Lemma7}.
Indeed, we have
\begin{lemma}\label{Lemma8} For $\alpha\geq1$, then
\begin{equation}\aligned\nonumber
\min_{z\in\overline{\mathcal{D}_{{\mathcal{G}}}}}\Big(2\pi y\alpha^2-\alpha-2y\frac{\vartheta_{XY}(\frac{y}{\alpha};x)}{-\vartheta_Y(\frac{y}{\alpha};x)}\Big)
\;\;\hbox{is achieved at}\;\;x=\frac{1}{2},y=\frac{\sqrt3}{2}.
\endaligned\end{equation}
\end{lemma}
By Lemma \ref{Lemma8}, after calculating a particular value at $x=\frac{1}{2},y=\frac{\sqrt3}{2}$, one has
\begin{lemma}\label{Lemma8a} For $z\in\overline{\mathcal{D}_{{\mathcal{G}}}}, \alpha\geq1.092212127$, then
\begin{equation}\aligned\nonumber
2\pi y\alpha^2-\alpha-2y\frac{\vartheta_{XY}(\frac{y}{\alpha};x)}{-\vartheta_Y(\frac{y}{\alpha};x)}
>0.
\endaligned\end{equation}
\end{lemma}

\begin{proof} It is split into three cases to complete the proof.
Note that $z\in\overline{\mathcal{D}_{{\mathcal{G}}}}$ implies that $y\geq\frac{\sqrt3}{2}$.

{\bf Case a:} $\{\frac{y}{\alpha}<\frac{\pi}{4}\}$. Using the upper bound of $\frac{\vartheta_{XY}(\frac{y}{\alpha};x)}{-\vartheta_Y(\frac{y}{\alpha};x)}$(Lemma \eqref{Lemma2c}),
\begin{equation}\aligned\nonumber
2\pi y\alpha^2-\alpha-2y\frac{\vartheta_{XY}(\frac{y}{\alpha};x)}{-\vartheta_Y(\frac{y}{\alpha};x)}
&>2\pi y\alpha^2-\alpha-3\alpha(1+\frac{\pi}{6}\frac{\alpha}{y})\\
&=\alpha\big(\frac{\pi\alpha}{2y}(4y^2-1)-4\big)\\
&>2\alpha\big(4y^2-3\big)\geq0.
\endaligned\end{equation}

{\bf Case b:} $\{\frac{y}{\alpha}\geq\frac{\pi}{4}\}\cap\{y\geq1\}$. Using the upper bound of $\frac{\vartheta_{XY}(\frac{y}{\alpha};x)}{-\vartheta_Y(\frac{y}{\alpha};x)}$(Lemma \eqref{Lemma2b}),
\begin{equation}\aligned\nonumber
2\pi y\alpha^2-\alpha-2y\frac{\vartheta_{XY}(\frac{y}{\alpha};x)}{-\vartheta_Y(\frac{y}{\alpha};x)}
&>2\pi y\alpha^2-\alpha-2\pi y(1+\epsilon_1)\\
&\geq2\pi \alpha^2-\alpha-2\pi (1+\epsilon_1)\\
&>0\;\;\hbox{if}\;\;\alpha>\frac{\sqrt{1+16\pi^2(1+\epsilon_1)}-1}{4\pi}.
\endaligned\end{equation}

\begin{equation}\aligned\nonumber
\epsilon_1:=\frac{\sum_{n=2}^\infty (n^4-n^2)e^{-\pi(n^2-1)\frac{y}{\alpha}}}
{1+\sum_{n=2}^\infty n^2e^{-\pi(n^2-1)\frac{y}{\alpha}}}\leq 13 e^{-\frac{3\pi^2}{4}}=0.007928797236\cdots.
\endaligned\end{equation}
And then
\begin{equation}\aligned\nonumber
\frac{\sqrt{1+16\pi^2(1+\epsilon_1)}-1}{4\pi}\leq1.086682914\cdots.
\endaligned\end{equation}

{\bf Case c:} $\{\frac{y}{\alpha}\geq\frac{\pi}{4}\}\cap\{y\in[\frac{\sqrt3}{2},1]\}$.
Using the function $\frac{\vartheta_{XY}(\frac{y}{\alpha};x)}{-\vartheta_Y(\frac{y}{\alpha};x)}$ is monotonically decreasing for $x\in[0,\frac{1}{2}]$.

\begin{equation}\aligned\label{KK}
2\pi y\alpha^2-\alpha-2y\frac{\vartheta_{XY}(\frac{y}{\alpha};x)}{-\vartheta_Y(\frac{y}{\alpha};x)}
&\geq2\pi y\alpha^2-\alpha-2y\frac{\vartheta_{XY}(\frac{y}{\alpha};\sqrt{1-y^2})}{-\vartheta_Y(\frac{y}{\alpha};\sqrt{1-y^2})}.
\endaligned\end{equation}
The function $\frac{\vartheta_{XY}(\frac{y}{\alpha};x)}{-\vartheta_Y(\frac{y}{\alpha};x)}$ is very close to the constant $\pi$ and $2\pi y\alpha^2-\alpha-2y\frac{\vartheta_{XY}(\frac{y}{\alpha};\sqrt{1-y^2})}{-\vartheta_Y(\frac{y}{\alpha};\sqrt{1-y^2})}$ is actually increasing for $y\in[\frac{\sqrt3}{2},1]$.
Continuing by \eqref{KK}, then
\begin{equation}\aligned\label{KKa}
2\pi y\alpha^2-\alpha-2y\frac{\vartheta_{XY}(\frac{y}{\alpha};x)}{-\vartheta_Y(\frac{y}{\alpha};x)}
&\geq2\pi y\alpha^2-\alpha-2y\frac{\vartheta_{XY}(\frac{y}{\alpha};\sqrt{1-y^2})}{-\vartheta_Y(\frac{y}{\alpha};\sqrt{1-y^2})}\\
&\geq\sqrt3\pi \alpha^2-\alpha-\sqrt3\frac{\vartheta_{XY}(\frac{\sqrt3}{2\alpha};\frac{1}{2})}{-\vartheta_Y(\frac{\sqrt3}{2\alpha};\frac{1}{2})}\\
&>0\;\;\hbox{if}\;\;\alpha>1.092212127.
\endaligned\end{equation}

\end{proof}

A slightly modification of Lemma \ref{Lemma8a} and application to Lemma \ref{Lemma7}, lead to that
\begin{lemma}\label{Lemma9} For $y\geq\frac{\sqrt3}{2},\alpha\geq1.092212127$ and $x\in\R$, then
\begin{equation}\aligned\nonumber
L(\alpha,x,y)
>0.
\endaligned\end{equation}
\end{lemma}
Therefore, Lemmas \ref{Lemma6}, \ref{Lemma9} and \eqref{Jx} provide the proof of Proposition \ref{Prop3}.
The proof is complete.

\section{Second order estimates}
Recall that
\begin{equation}\aligned\nonumber
\Gamma_b=\{
z\in\mathbb{H}: z=e^{i\theta},\; \theta\in[\frac{\pi}{3},\frac{\pi}{2}]
\}.
\endaligned\end{equation}
In this Section, we aim to prove that
\begin{proposition}\label{PropA1} For $\alpha\in[0.9155730607,1]$,
\begin{equation}\aligned\nonumber
\min_{z\in\mathbb{H}}M(\alpha,z)=\min_{z\in\overline{\mathcal{D}_{\mathcal{G}}}}M(\alpha,z)=\min_{z\in\Gamma_b}M(\alpha,z).
\endaligned\end{equation}
\end{proposition}
A combination of Propositions \ref{Prop1} and \ref{PropA1} gives that
\begin{theorem}\label{ThA}
For $\alpha\in(0,1]$,
\begin{equation}\aligned\nonumber
\min_{z\in\mathbb{H}}M(\alpha,z)=\min_{z\in\overline{\mathcal{D}_{\mathcal{G}}}}M(\alpha,z)=\min_{z\in\Gamma}M(\alpha,z).
\endaligned\end{equation}

\end{theorem}
For the proof of Proposition \ref{PropA1}, we have to split it into two cases.

In {\bf case $(a)$: i.e., $z\in\mathcal{D}_{\mathcal{G}}\cap \{y\geq 2\}$}, we shall prove that
\begin{equation}\aligned\nonumber
\frac{\partial}{\partial y}M(\alpha,z)>0,\;\;\hbox{for}\;\;z\in\mathcal{D}_{\mathcal{G}}\cap \{y\geq 2\}.
\endaligned\end{equation}
Then it follows that
\begin{equation}\aligned\label{Ka}
\min_{z\in\overline{\mathcal{D}_{\mathcal{G}}}\cap \{y\geq 2\}}M(\alpha,z)&=\min_{z\in\overline{\mathcal{D}_{\mathcal{G}}}\cap \{y=2\}}M(\alpha,z)\\
\min_{z\in\mathbb{H}}M(\alpha,z)=\min_{z\in\overline{\mathcal{D}_{\mathcal{G}}}}M(\alpha,z)&=\min_{z\in\overline{\mathcal{D}_{\mathcal{G}}}\cap \{y\leq2\}}M(\alpha,z)
\endaligned\end{equation}
It then reduces the minimization from the half fundamental domain to a small finite region $\overline{\mathcal{D}_{\mathcal{G}}}\cap \{y\leq2\}$. This case will be proved in Proposition \ref{PropA2}.

In {\bf case $(b)$: i.e., $z\in\mathcal{D}_{\mathcal{G}}\cap \{y\leq 2\}$}, we then establish a minimum principle as follows.
\begin{proposition}[A minimum principle]\label{PropA}  Assume that $\mathcal{W}$ is modular invariant, i.e.,
 \begin{equation}\aligned\label{M999}
\mathcal{W}(\frac{az+b}{cz+d})=\mathcal{W}(z),\;\;\hbox{for all}\;\;\left(
                                                                      \begin{array}{cc}
                                                                        a & b \\
                                                                        c & d \\
                                                                      \end{array}
                                                                    \right)\in \hbox{SL}_2(\mathbb{Z}),
\endaligned\end{equation}
and
 \begin{equation}\aligned\nonumber
\mathcal{W}(-\overline{z})=\mathcal{W}(z).
\endaligned\end{equation}

If \begin{equation}\aligned\label{Cabc1}
&(\frac{\partial^2}{\partial y^2}+\frac{2}{y}\frac{\partial}{\partial y})\mathcal{W}(z)>0,\;\;z\in\mathcal{D}_{\mathcal{G}}\cap\{y\leq y_0\}
\;\;\hbox{for some}\;\;y_0\geq1
\endaligned\end{equation}
and
\begin{equation}\aligned\label{Cabc2}
&\frac{\partial^2}{\partial y\partial x}\mathcal{W}(z)>0,\;\;z\in\mathcal{D}_{\mathcal{G}}\cap\{y\leq1\}.
\endaligned\end{equation}
Here $\mathcal{D}_{\mathcal{G}}$ is the fundamental domain corresponding to modular group $\hbox{SL}_2(\mathbb{Z})$, explicitly, $\mathcal{D}_{\mathcal{G}}=\{
z\in\mathbb{H}: |z|>1,\; 0<x<\frac{1}{2}
\}.$
Then
\begin{equation}\aligned\nonumber
\min_{z\in\overline{\mathcal{D}_{\mathcal{G}}}\cap\{y\leq y_0\}}\mathcal{W}(z)=\min_{z\in\Gamma_b}\mathcal{W}(z).
\endaligned\end{equation}

\end{proposition}
\begin{remark}
\begin{itemize} Four remarks are listed in order:
  \item The condition $(\frac{\partial^2}{\partial y^2}+\frac{2}{y}\frac{\partial}{\partial y})\mathcal{W}(z)>0,\;\;z\in\mathcal{D}_{\mathcal{G}}\cap\{y\leq y_0\}$ can be replaced by
$\frac{\partial^2}{\partial y^2}\mathcal{W}(z)>0,\;\;z\in\mathcal{D}_{\mathcal{G}}\cap\{y\leq y_0\}$, the conclusion does not change;
  \item The condition $\frac{\partial^2}{\partial y\partial x}\mathcal{W}(z)>0,\;\;z\in\mathcal{D}_{\mathcal{G}}\cap\{y\leq1\}$ is almost sharp in the sense that the region $\mathcal{D}_{\mathcal{G}}\cap\{y\leq1\}$ for $\frac{\partial^2}{\partial y\partial x}\mathcal{W}(z)>0$  is optimal;
  \item $\Gamma_b$ is a partial boundary of $\overline{\mathcal{D}_{\mathcal{G}}}$$($arc part$)$;
  \item In general, the conditions \eqref{Cabc1} and \eqref{Cabc2} are natural for certain classes of modular invariant functions, hence Proposition
 \ref{PropA} is quite useful in application. Therefore, we formulate a minimum principle here. It is a minimum principle for
 modular invariant functions.
\end{itemize}

\end{remark}

We will show the second order estimates for $M(\alpha,z)$ on the forms of \eqref{Cabc1} and \eqref{Cabc2} in Lemmas \ref{Lemma37} and \ref{LemmaMxy} respectively. Then by Proposition \ref{PropA},
\begin{equation}\aligned\label{Kb}
\min_{z\in\overline{\mathcal{D}_{\mathcal{G}}}\cap \{y\leq2\}}M(\alpha,z)
=\min_{z\in\Gamma_b}M(\alpha,z).
\endaligned\end{equation}
Therefore, \eqref{Ka} and \eqref{Kb} yield Proposition \ref{PropA1}. It remains to prove Propositions \ref{PropA2} and \ref{PropA}, Lemmas \ref{Lemma37} and \ref{LemmaMxy}. We shall prove these in order in the rest of this Section.

\begin{proposition}\label{PropA2} For $\alpha\in[0.9155730607,1]$,
\begin{equation}\aligned\nonumber
\frac{\partial}{\partial y}M(\alpha,z)>0,\;\;\hbox{for}\;\;z\in\mathcal{D}_{\mathcal{G}}\cap \{y\geq 2\}.
\endaligned\end{equation}
\end{proposition}

To prove Proposition \ref{PropA2}, we regroup $\frac{\partial}{\partial y}M(\alpha,z)$ in using of Lemma \ref{Lemma3}.

\begin{lemma} An exponentially decaying identity for $\partial_{y} M(\alpha,z)$.
\begin{equation}\aligned\nonumber
\partial_{y} M(\alpha,z)
=&\frac{1}{\pi\sqrt{\frac{y}{\alpha}}\alpha^4}
\sum_{n\in\mathbb{Z}}
\Big(
2\alpha y\vartheta_X(\frac{y}{\alpha};nx)
+y^2\vartheta_{XX}(\frac{y}{\alpha};nx)\\
&-(\alpha^4\pi^2 y^2 n^4-\pi \alpha^3y n^2-\frac{1}{4}\alpha^2)\vartheta(\frac{y}{\alpha};nx)
\Big)\cdot e^{-\pi \alpha y n^2}
\endaligned\end{equation}
\end{lemma}
For convenience, we set that
\begin{equation}\aligned\nonumber
P_n(z,\alpha):&=\Big(
2\alpha y\vartheta_X(\frac{y}{\alpha};nx)
+y^2\vartheta_{XX}(\frac{y}{\alpha};nx)
-(\alpha^4\pi^2 y^2 n^4-\pi \alpha^3y n^2-\frac{1}{4}\alpha^2)\vartheta(\frac{y}{\alpha};nx)
\Big)\cdot e^{-\pi \alpha y n^2}.
\endaligned\end{equation}
Then
\begin{equation}\aligned\label{Myp}
\partial_{y} M(\alpha,z)
=&\frac{1}{\pi\sqrt{\frac{y}{\alpha}}\alpha^4}
\sum_{n\in\mathbb{Z}}
P_n(z,\alpha)=\frac{1}{\pi\sqrt{\frac{y}{\alpha}}\alpha^4}\Big(P_0(z,\alpha)+2P_1(z,\alpha)+
\sum_{n\geq3}
2P_n(z,\alpha)\Big).
\endaligned\end{equation}
Then the $P_0(z,\alpha)+2P_1(z,\alpha)$ and $\sum_{n\geq3}
2P_n(z,\alpha)$ are the major and error terms of $\partial_{y} M(\alpha,z)$ after a positive factor. We shall estimate these two terms
in Lemmas \ref{Lemma32} and \ref{Lemma33} respectively.

\begin{lemma}\label{Lemma32} For $(\alpha,y)\in[0.9155730607,1]\times[2,\infty)$,
\begin{equation}\aligned\nonumber
P_0(z,\alpha)+2P_1(z,\alpha)
\geq \alpha^2(\frac{1}{4}-(\pi\alpha^2 y^2-\pi\alpha y-\frac{1}{4}) e^{-\pi\alpha y})\vartheta(\frac{y}{\alpha};0)
-(8\pi\alpha y+8\pi^2 y^2) e^{-\pi\frac{y}{\alpha}}>0.1.
\endaligned\end{equation}
\end{lemma}

\begin{lemma}\label{Lemma33} For $(\alpha,y)\in[0.9155730607,1]\times[2,\infty)$,
\begin{equation}\aligned\nonumber
\mid\sum_{n\geq3}
2P_n(z,\alpha)\mid
\leq
\sum_{n=2}^\infty 4(4\pi\alpha y+4\pi^2 y^2+\alpha^4\pi^2 y^2 n^4)e^{-\pi\alpha y n^2}\leq4\cdot 10^{-7}.
\endaligned\end{equation}
\end{lemma}
In view of \eqref{Myp}, by Lemmas \ref{Lemma32} and \ref{Lemma33}, one gets
Proposition \ref{PropA2}.

Next, we are going to prove the minimum principle(Proposition \ref{PropA}). Before this, we shall introduce one definition and two lemmas.

\begin{definition}\label{DefA} A function $f(z)$ is called modular invariant if it is invariant under $SL(2,\mathbb{Z})$ and $z\mapsto-\overline{z}$.
Namely,
 \begin{equation}\aligned\nonumber
\mathcal{W}(\frac{az+b}{cz+d})=\mathcal{W}(z),\;\;\hbox{for all}\;\;\left(
                                                                      \begin{array}{cc}
                                                                        a & b \\
                                                                        c & d \\
                                                                      \end{array}
                                                                    \right)\in \hbox{SL}_2(\mathbb{Z})
\endaligned\end{equation}
and
 \begin{equation}\aligned\nonumber
\mathcal{W}(-\overline{z})=\mathcal{W}(z).
\endaligned\end{equation}

\end{definition}
Many physically related functions(like strain functions in crystals) are modular invariant according to Definition \ref{DefA}.
\begin{example} Modular invariant functions generated by  a single variable function on lattice are
\begin{equation}\aligned\nonumber
\sum_{\mathbb{P}\in L\setminus\{0\}, |L|=1} f(\alpha|\mathbb{P}|^2)\;\;\hbox{or}\;\;
\sum_{(m,n)\in\mathbb{Z}^2\setminus\{0\} }
f(\alpha\frac{|mz+n|^2}{\Im(z) }).
\endaligned\end{equation}
Here $f$ is any function with $f(x)=O(x^{-(2+\delta)})$ with $\delta>0$ when $x$ large enough.

\end{example}

There are several basic properties satisfied by modular invariant functions, we collected them there.

The first lemma is about the information on $\Gamma_b$, and the second lemma is the properties on $x-$axis and $\frac{1}{2}-$axis.
A proof can be found in \cite{Mon1988}(with slight modification), we omit the details here.
\begin{lemma}\label{Lemma3.3} If $\mathcal{W}(z)$ is a modular invariant function, then for $z=e^{i\theta}\in\Gamma_b$

\begin{itemize}
  \item $
\frac{\partial}{\partial x}\mathcal{W}( e^{i\theta})\cos(\theta)=-\frac{\partial}{\partial y}\mathcal{W}( e^{i\theta})\sin(\theta).
$
  \item $
\hbox{either}\;\;\frac{\partial}{\partial x}\mathcal{W}(z)\mid_{z\in\Gamma_b}\geq0\;\;\hbox{or}\;\;\frac{\partial}{\partial y}\mathcal{W}(z)\mid_{z\in\Gamma_b}\geq0.
$
\end{itemize}

\end{lemma}

\begin{lemma}\label{Lemma3.4}If $\mathcal{W}(z)$ is a modular invariant function, then

\begin{itemize}
  \item $\mathcal{W}(iy)=\mathcal{W}(i\frac{1}{y}),\;\; \mathcal{W}(\frac{1}{2}+i\frac{y}{2})=\mathcal{W}(\frac{1}{2}+i\frac{1}{2y}).$
  \item $\frac{\partial}{\partial y}\mathcal{W}(iy)\mid_{y=1}=0, \;\; \frac{\partial}{\partial y}\mathcal{W}(\frac{1}{2}+iy)\mid_{y=\frac{\sqrt3}{2}}=0.$
\end{itemize}

\end{lemma}

With these prepare work, we are going to prove Proposition \ref{PropA}.

\begin{proof}{\bf Proof of Proposition \ref{PropA}.}
By Lemmas \ref{Lemma3.3}-\ref{Lemma3.4}, it follows that
\begin{equation}\aligned\label{Ma}
\frac{\partial}{\partial y}\mathcal{W}(z)\mid_{z=i}=0,\;\;\frac{\partial}{\partial y}\mathcal{W}(z)\mid_{z=e^{i\frac{\pi}{3}}}=0
\endaligned\end{equation}
and
\begin{equation}\aligned\label{Mb}
\hbox{for any} \;\;z\in \Gamma_b,\;\;
\hbox{either}\;\; \frac{\partial}{\partial x}\mathcal{W}(z)\geq0,\;\;
  \hbox{or}\;\; \frac{\partial}{\partial y}\mathcal{W}(z)\geq0.
\endaligned\end{equation}
From $\frac{\partial}{\partial y}\mathcal{W}(z)\mid_{z=i}=0$ in \eqref{Ma}, by \eqref{Cabc2},
\begin{equation}\aligned\label{GT1}
\frac{\partial}{\partial y}\mathcal{W}(z)>0,\;\;\hbox{ for}\;\; z\in\mathcal{D}_{\mathcal{G}}\cap \{y=1\}.
\endaligned\end{equation}
Thus by \eqref{Cabc1}(note that here $\frac{\partial^2}{\partial y^2}+\frac{2}{y}\frac{\partial}{\partial y}=y^{-2}\frac{\partial}{\partial y}(y^2\frac{\partial}{\partial y})$),
\begin{equation}\aligned\label{GT2}
\frac{\partial}{\partial y}\mathcal{W}(z)>0,\;\;\hbox{ for}\;\; z\in\mathcal{D}_{\mathcal{G}}\cap \{1\leq y\leq y_0\}.
\endaligned\end{equation}
By \eqref{GT2},
\begin{equation}\aligned\label{Pj1}
\min_{z\in\overline{\mathcal{D}_{\mathcal{G}}}\cap\{y\leq y_0\}}\mathcal{W}(z)
=\min_{z\in\overline{\mathcal{D}_{\mathcal{G}}}\cap\{y\leq1\}}\mathcal{W}(z).
\endaligned\end{equation}
Next, we shall show that $\min_{z\in\overline{\mathcal{D}_{\mathcal{G}}}\cap\{y\leq1\}}\mathcal{W}(z)$ must be attained at $\Gamma_b$. If not, assume that the minimizer is attained at some ${z_0:=x_0+iy_0\in\overline{\mathcal{D}_{\mathcal{G}}}\cap\{y<1\}}\setminus \Gamma_b$. Then
by Fermat's Theorem,
\begin{equation}\aligned\label{CK1}
\frac{\partial}{\partial x}\mathcal{W}(z)\mid_{z=z_0}=0,\;\;\frac{\partial}{\partial y}\mathcal{W}(z)\mid_{z=z_0}=0.
\endaligned\end{equation}
To get a contradiction, we consider a point $x_0+i\sqrt{1-x_0^2}\in\Gamma_b$ with the same $x-$coordinate of $z_0$.
Comparing the values at $z_0$ and $x_0+i\sqrt{1-x_0^2}$, by \eqref{Mb}, \eqref{Cabc1} and \eqref{Cabc2}, one obtains that

\begin{equation}\aligned\label{Mb2}
\hbox{either}\;\; \frac{\partial}{\partial x}\mathcal{W}(z)\mid_{z=z_0}>0,\;\;
  \hbox{or}\;\; \frac{\partial}{\partial y}\mathcal{W}(z)\mid_{z=z_0}>0.
\endaligned\end{equation}
Then, there is a contradiction between \eqref{CK1} and \eqref{Mb2}. Therefore,
\begin{equation}\aligned\label{Pj2}
\min_{z\in\overline{\mathcal{D}_{\mathcal{G}}}\cap\{y\leq1\}}\mathcal{W}(z)
=\min_{z\in\Gamma_b}\mathcal{W}(z).
\endaligned\end{equation}
Therefore, \eqref{Pj1} and \eqref{Pj2} yield the result.

\end{proof}

In the end of this Section, we are going to prove Lemmas \ref{Lemma37} and \ref{LemmaMxy}.
\begin{lemma}\label{Lemma37} For $\alpha\in[0.9155730607,1]$,
\begin{equation}\aligned\nonumber
(\partial_{yy}+\frac{2}{y}\partial_y)M(\alpha,z)>0,\;\;\hbox{for}\;\;z=x+iy\in\{[0,\frac{1}{2}]\times[\frac{\sqrt3}{2},2]\}.
\endaligned\end{equation}
\end{lemma}

Using the explicit expression of $M(\alpha,z)$ in \eqref{J}, one has

\begin{lemma} A double sum identity for $(\partial_{yy}+\frac{2}{y}\partial_y)M(\alpha,z)$ with exponential factors.
\begin{equation}\aligned\nonumber
(\partial_{yy}+\frac{2}{y}\partial_y)M(\alpha,z)=&
\sum_{n,m}
P_{m,n}(x,y),
\endaligned\end{equation}
where
\begin{equation}\aligned\nonumber
P_{m,n}(x,y):=&
\Big(
\pi^2 \alpha^2(n^2-\frac{(m+nx)^2}{y^2})^2(yn^2+\frac{(m+nx)^2}{y})+\frac{2n^2}{y}\\
&-2\pi \alpha(n^2-\frac{(m+nx)^2}{y^2})^2-\frac{2\pi \alpha}{y}n^2(yn^2+\frac{(m+nx)^2}{y})
\Big)
e^{-\pi \alpha(yn^2+\frac{(m+nx)^2}{y})}.
\endaligned\end{equation}
\end{lemma}
To illustrate the proof, we define
\begin{equation}\aligned\nonumber
P_A(x,y):=\sum_{|n|\leq2,|m|\leq2}P_{m,n}(x,y),\;\;P_B(x,y):=\sum_{|n|\geq3\;\hbox{or}\;|m|\geq3}P_{m,n}(x,y)
\endaligned\end{equation}
be the approximate and error parts of $(\partial_{yy}+\frac{2}{y}\partial_y)M(\alpha,z)$ respectively.
Then
\begin{equation}\aligned\label{Mxy}
(\partial_{yy}+\frac{2}{y}\partial_y)M(\alpha,z)=P_A(x,y)+P_B(x,y).
\endaligned\end{equation}
With this split, we find that the error part $P_B(x,y)$ is negligible comparing the approximate part $P_A(x,y)$ in the desired domain. Namely,
$$
\frac{|P_B(x,y)|}{P_A(x,y)}\leq 10^{-3}\ll1,\;\;\hbox{for}\;\; (\alpha,x,y)\in[0.9155730607,1]\times[0,\frac{1}{2}]\times[\frac{\sqrt3}{2},2].
$$

Precisely, one can show that

\begin{lemma}\label{PropA3} For $\alpha\in[0.9155730607,1]$, $z=x+iy\in\{[0,\frac{1}{2}]\times[\frac{\sqrt3}{2},2]\}$, it holds that
\begin{equation}\aligned\nonumber
P_A(x,y)\geq 10^{-3},\;\; |P_B(x,y)|\leq 10^{-6}.
\endaligned\end{equation}
\end{lemma}

Therefore, Lemma \ref{PropA3} and identity \eqref{Mxy} prove \ref{Lemma37}.

It remains to prove second mixed order estimate(Lemma \ref{LemmaMxy}).

Using Lemma \ref{Lemma3}, after direct computations and regrouping the terms, one gets
\begin{lemma}\label{Lemma38} An exponentially decaying identity for $\partial_{xy} M(\alpha,z)$.
\begin{equation}\aligned\nonumber
\partial_{xy} M(\alpha,z)
=&\frac{2}{\pi\sqrt{\frac{y}{\alpha}}\alpha^4}\Big(
\sum_{n=1}^\infty(\pi \alpha^3 yn^3+\frac{1}{4}a^2 n-a^4 n^5\pi^2 y)\cdot\vartheta_Y(\frac{y}{\alpha};nx)e^{-\pi \alpha y n^2}\\
&+\sum_{n=1}^\infty 2\alpha yn\cdot\vartheta_{XY}(\frac{y}{\alpha};nx)e^{-\pi \alpha y n^2}
+\sum_{n=1}^\infty y^2n\cdot\vartheta_{XXY}(\frac{y}{\alpha};nx)e^{-\pi \alpha y n^2}
\Big).
\endaligned\end{equation}

\end{lemma}

To estimate second mixed order derivative $\partial_{xy} M(\alpha,z)$, we further deform the expression in Lemma \ref{Lemma38} as follows

\begin{lemma}\label{Lemma3.7} A simplified form of $\partial_{xy} M(\alpha,z)$ after a positive factor.
\begin{equation}\aligned\nonumber
\partial_{xy} M(\alpha,z)
=&\frac{2}{\pi\sqrt{\frac{y}{\alpha}}\alpha^2}\cdot(-\vartheta_Y(\frac{y}{\alpha};x))e^{-\pi\alpha y}
\cdot\Big(
M_A(\alpha,z)+M_B(\alpha,z)
\Big).
\endaligned\end{equation}

\end{lemma}

Note that by Lemma \ref{Lemma6}, the factor $\frac{2}{\pi\sqrt{\frac{y}{\alpha}}\alpha^2}\cdot(-\vartheta_Y(\frac{y}{\alpha};x))e^{-\pi\alpha y}$ in Lemma \ref{Lemma3.7} is positive.

To analyze $\partial_{xy} M(\alpha,z)$, we denote that

\begin{equation}\aligned\nonumber
M_A(\alpha,z)
:&=2\frac{y}{\alpha} \cdot(-\frac{\vartheta_{XY}(\frac{y}{\alpha};x)}{\vartheta_{Y}(\frac{y}{\alpha};x)})
-(\frac{y}{\alpha})^2\cdot\frac{\vartheta_{XXY}(\frac{y}{\alpha};x)}{\vartheta_{Y}(\frac{y}{\alpha};x)}
+(\alpha^2\pi^2 y^2-\pi \alpha y-\frac{1}{4})\\
&+\Big(
(\alpha^2 2^5\pi^2 y^2-8\pi \alpha y -\frac{1}{2})\frac{\vartheta_{Y}(\frac{y}{\alpha};2x)}{\vartheta_{Y}(\frac{y}{\alpha};x)}\\
&-4\frac{y}{\alpha} \frac{\vartheta_{XY}(\frac{y}{\alpha};2x)}{\vartheta_{Y}(\frac{y}{\alpha};x)}
-2(\frac{y}{\alpha})^2 \frac{\vartheta_{XXY}(\frac{y}{\alpha};2x)}{\vartheta_{Y}(\frac{y}{\alpha};x)}
\Big)\cdot e^{-3\pi \alpha y }
,\\
M_B(\alpha,z)
:&=\sum_{n=3}^\infty
\Big(
(\alpha^2 n^5\pi^2 y^2-\pi \alpha y n^3-\frac{1}{4}n )\frac{\vartheta_{Y}(\frac{y}{\alpha};nx)}{\vartheta_{Y}(\frac{y}{\alpha};x)}\\
&-2\frac{y}{\alpha} n\frac{\vartheta_{XY}(\frac{y}{\alpha};nx)}{\vartheta_{Y}(\frac{y}{\alpha};x)}
-(\frac{y}{\alpha})^2n \frac{\vartheta_{XXY}(\frac{y}{\alpha};nx)}{\vartheta_{Y}(\frac{y}{\alpha};x)}
\Big)\cdot e^{-\pi \alpha y (n^2-1)}
\endaligned\end{equation}
being the major and error terms in estimating $\partial_{xy} M(\alpha,z)$.

In the range of $(x,\frac{y}{\alpha})\in[0,\frac{1}{2}]\times[\frac{\sqrt3}{2},\infty)$, one roughly has
$$\frac{\vartheta_{Y}(\frac{y}{\alpha};nx)}{\vartheta_{Y}(\frac{y}{\alpha};x)}\cong n (n\in \mathbb{N}^+),\;\;
-\frac{\vartheta_{XY}(\frac{y}{\alpha};nx)}{\vartheta_{Y}(\frac{y}{\alpha};x)}\cong n\pi,\;\;
\frac{\vartheta_{XXY}(\frac{y}{\alpha};nx)}{\vartheta_{Y}(\frac{y}{\alpha};x)}\cong n\pi^2.
$$
A precise version of these bounds can be found in Lemmas \ref{Lemma2a}-\ref{Lemma2c}.
Then approximately,
\begin{equation}\aligned\nonumber
M_A(\alpha,z)
\cong&2\pi\frac{y}{\alpha}
-\pi^2(\frac{y}{\alpha})^2
+(\alpha^2\pi^2 y^2-\pi \alpha y-\frac{1}{4})\\
&+\Big(
(\alpha^2 2^6\pi^2 y^2-16\pi \alpha y -1)-8\pi\frac{y}{\alpha}
-4\pi^2(\frac{y}{\alpha})^2
\Big)\cdot e^{-3\pi \alpha y }\\
|M_B(\alpha,z)|
\cong&\alpha^2 3^6\pi^2 y^2 e^{-8\pi \alpha y }.
\endaligned\end{equation}

Indeed, one can show that
\begin{lemma}\label{Lemma3.8} For $(\alpha,x,y)\in[0.9155730607,1]\times[0,\frac{1}{2}]\times[\frac{\sqrt3}{2},1.05]$, it holds that
\begin{equation}\aligned\nonumber
M_A(\alpha,z)\geq10^{-1},\;\;|M_B(\alpha,z)|\leq2\cdot 10^{-5}.
\endaligned\end{equation}
Then
\begin{equation}\aligned\nonumber
M_A(\alpha,z)+M_B(\alpha,z)>0.
\endaligned\end{equation}

\end{lemma}

By Lemmas \ref{Lemma3.7} and \ref{Lemma3.8}, one has
\begin{lemma}\label{LemmaMxy} For $(\alpha,x,y)\in[0.9155730607,1]\times[0,\frac{1}{2}]\times[\frac{\sqrt3}{2},1.05]$,
\begin{equation}\aligned\nonumber
\partial_{xy} M(\alpha,z)\geq0.
\endaligned\end{equation}
\end{lemma}


\section{The analysis on $\Gamma_a$}

Note that by Propositions \ref{Prop1} and \ref{PropA1}, we obtain that
\begin{equation}\aligned\nonumber
\min_{z\in\mathbb{H}} M(\alpha,z)=\min_{z\in \Gamma} M(\alpha,z)\;\;\hbox{for}\;\; \alpha\in(0,1].
\endaligned\end{equation}
See Theorem \ref{ThA}. It remains to determine the minimizers of $M(\alpha,z)$ on the boundary $\Gamma$.
$\Gamma$ consists of two parts $\Gamma_a$ and $\Gamma_b$. We shall analyze $M(\alpha,z)$ on $\Gamma_a$ and $\Gamma_b$
 in this and next section respectively.

Recall that
\begin{equation}\aligned\nonumber
\Gamma_a=\{
z\in\mathbb{H}: \Re(z)=0,\; \Im(z)\geq1
\}.
\endaligned\end{equation}

In this section, we aim to establish that
\begin{theorem}\label{Th31} Assume that $\alpha\in(0,1]$, then
\begin{equation}\aligned
\min_{z\in\Gamma_a}M(\alpha,z)=
\begin{cases}
\hbox{is achieved at}\;\;i, &\hbox{if}\;\; \alpha\in[\alpha_1,1],\\
\hbox{is achieved at}\;\;y_\alpha i, &\hbox{if}\;\; \alpha\in(0,\alpha_1).
\end{cases}
\endaligned\end{equation}
Here $\alpha_1=0.8947042694\cdots$ and $y_\alpha>1$.

$\frac{1}{\alpha_1}$ is the unique solution of
\begin{equation}\aligned\nonumber
\theta_{yy}(\alpha,i)=\pi\alpha M_{yy}(\alpha,i)\;\;\hbox{for}\;\;\alpha\in[1,\frac{9}{8}],
\endaligned\end{equation}
or equivalently,
\begin{equation}\aligned\nonumber
\frac{2}{\pi\alpha}\frac{\sum_{n,m}(2m^2-\pi\alpha(n^2-m^2)^2)e^{-\pi\alpha(n^2+m^2)}}
{\sum_{n,m}(n^4+3m^4-\pi\alpha(n^4-m^4)(n^2-m^2))e^{-\pi\alpha(n^2+m^2)}
}=1\;\;\hbox{for}\;\;\alpha\in[1,\frac{9}{8}].
\endaligned\end{equation}

\end{theorem}

The proof of of Theorem \ref{Th31} consists of two cases, the case $a\in(0,\frac{100}{101}]$ and the case $a\in[\frac{100}{101},1]$.
The case $a\in[\frac{100}{101},1]$ is given in previous section.
In fact, by Lemmas \ref{Lemma3.4} and \ref{Lemma37}, one gets
\begin{lemma} For $\alpha\in[\frac{100}{101},1]$,
$$
\min_{z\in\Gamma_a}M(\alpha,z)\;\;\hbox{is achieved at}\;\;i.
$$
\end{lemma}

In this section, we focus on the case $a\in(0,\frac{100}{101}]$.
By the duality of the functionals(Lemma \ref{Lemma1}), the case $a\in(0,\frac{100}{101}]$ of Theorem \ref{Th31} is equivalent to

\begin{theorem}\label{Th32} Assume that $\alpha\in[\frac{101}{100},\infty)$, then
\begin{equation}\aligned
\min_{z\in\Gamma_a}\Big(\theta(\alpha,z)-\pi\alpha
M(\alpha,z)
\Big)
=
\begin{cases}
\hbox{is achieved at}\;\;i, &\hbox{if}\;\; \alpha\in[1,\frac{1}{\alpha_1}],\\
\hbox{is achieved at}\;\;y_\alpha i, &\hbox{if}\;\; \alpha\in(\frac{1}{\alpha_1},\infty).
\end{cases}
\endaligned\end{equation}
Here $\frac{1}{\alpha_1}=1.117687748\cdots$ is the same as in Theorem \ref{Th31}.
\end{theorem}

\begin{figure}
\centering
 \includegraphics[scale=0.33]{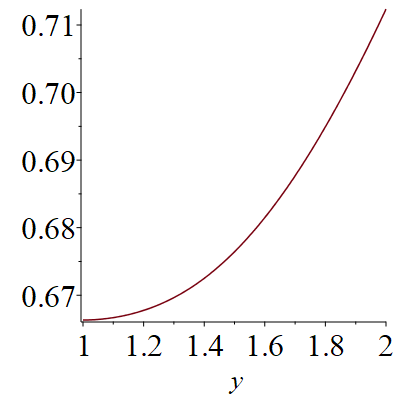}
 \includegraphics[scale=0.33]{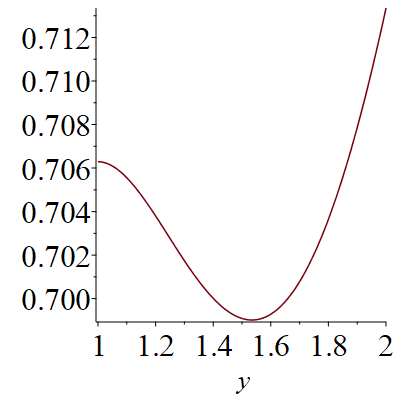}
 \includegraphics[scale=0.33]{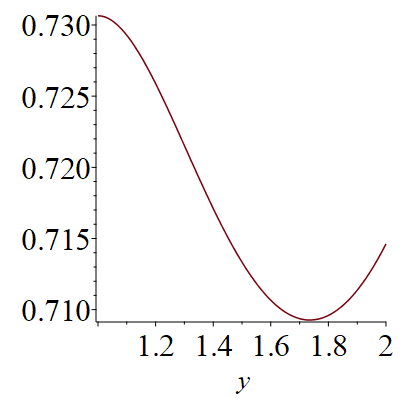}
 \caption{
 The shape of $\Big(\theta(\alpha,z)-\pi\alpha
M(\alpha,z)
\Big)$ on $\Gamma_a$ for various $\alpha\geq1$.
 }
\label{Yvv-Shape}
\end{figure}
We shall prove that $\Big(\theta(\alpha,z)-\pi\alpha
M(\alpha,z)
\Big)$ admits at most two critical points on $\Gamma_a$, an explicit critical point is $y=1$ independent of $\alpha$.
The emergence of the possible second critical point depends on the value of $\alpha$. We shall also determine the threshold
of $\alpha$ such that there is an another critical point beyond the threshold.

Note that $z\in\Gamma_a$ implies that $z=iy$, to locate the critical point of $\Big(\theta(\alpha,z)-\pi\alpha
M(\alpha,z)
\Big)$ on $\Gamma_a$, we deform that
\begin{equation}\aligned\label{J0y}
\Big(\theta_y(\alpha,iy)-\pi\alpha
M_y(\alpha,iy)
\Big)
=
\theta_y(\alpha,iy)\cdot
\Big(1-
\frac{\pi\alpha
M_y(\alpha,iy)}{\theta_y(\alpha,iy)}
\Big)
\endaligned\end{equation}
One observes that(See Lemma \ref{Lemma3.4})
\begin{lemma}\label{Lemma552} Assume that $\alpha\geq1$, then $\theta_y(\alpha,iy)\geq0$ for $y\geq1$ with $"="$ holds if only $y=1$.
\end{lemma}

To prove Theorem \ref{Th32}, we classify all the shapes of $(\theta(\alpha,z)-\pi\alpha M(\alpha,z))$ on $\Gamma_a$ for all $\alpha\geq\frac{101}{100}$.

\begin{proposition}\label{PropJb} Assume that $\alpha\geq\frac{101}{100}$.
$\Big(\theta(\alpha,z)-\pi\alpha
M(\alpha,z)
\Big)$ admits at most one critical point on $\Gamma_a$ except $y=1$ depending on the value of $\alpha$.
More precisely,
\begin{itemize}
  \item $(1)$ if $\theta_{yy}(\alpha,i)\geq\pi\alpha M_{yy}(\alpha,i)
$, then $\Big(\theta(\alpha,z)-\pi\alpha
M(\alpha,z)
\Big)$ is increasing on $\Gamma_a$, and
$$\min_{z\in\Gamma_a}\Big(\theta(\alpha,z)-\pi\alpha
M(\alpha,z)
\Big)\;\;
\hbox{is achieved at}\;\;i.
$$

  \item $(2)$ if $\theta_{yy}(\alpha,i)<\pi\alpha M_{yy}(\alpha,i)$, then $\Big(\theta(\alpha,z)-\pi\alpha
M(\alpha,z)
\Big)$ is first decreasing  then increasing on $\Gamma_a$, and
$$\min_{z\in\Gamma_a}\Big(\theta(\alpha,z)-\pi\alpha
M(\alpha,z)
\Big)\;\;
\hbox{is achieved at}\;\;iy_\alpha,\;\;\hbox{where}\;\; y_\alpha>1.
$$
\end{itemize}

\end{proposition}

The shapes of $\Big(\theta(\alpha,z)-\pi\alpha
M(\alpha,z)
\Big)$ on $\Gamma_a$ for various $\alpha$ are illustrated in Picture \ref{Yvv-Shape}.
We postpone the proof of Proposition \ref{PropJb}, and determine the threshold equation
$\theta_{yy}(\alpha,i)=\pi\alpha M_{yy}(\alpha,i)$ firstly. It is classified by

\begin{lemma}\label{Lemmayyy} Assume that $\alpha\geq\frac{101}{100}$. It holds that
\begin{itemize}
  \item $(1):$ $\theta_{yy}(\alpha,i)=\pi\alpha M_{yy}(\alpha,i)$ admits only one solution, i.e., $\alpha=1.117687748\cdots$;
  \item $(2):$ $\theta_{yy}(\alpha,i)\geq\pi\alpha M_{yy}(\alpha,i)$ $\Leftrightarrow$ $\alpha\leq1.117687748\cdots$;
  \item $(3):$ $\theta_{yy}(\alpha,i)<\pi\alpha M_{yy}(\alpha,i)$ $\Leftrightarrow$ $\alpha>1.117687748\cdots$.
\end{itemize}

\end{lemma}

Theorem \ref{Th32} is then followed by Proposition \ref{PropJb} and Lemma \ref{Lemmayyy}.
It remains to prove Proposition \ref{PropJb} and Lemma \ref{Lemmayyy}.

To prove Proposition \ref{PropJb}, by deformation \eqref{J0y} and Lemma \ref{Lemma552}, it suffices to prove

\begin{lemma}\label{Lemmaxxx} Assume that $\alpha\geq\frac{101}{100}$. It holds that

\begin{itemize}
  \item $(1):$ for $y\geq2\alpha$, $\theta_y(\alpha,iy)-\pi\alpha M_y(\alpha,iy)>0$;
  \item $(2):$ for $y\in[1,2\alpha]$, $\partial_y\frac{M_y(\alpha,iy)}{\theta_y(\alpha,iy)}\leq0$.
\end{itemize}

\end{lemma}

While to prove Lemma \ref{Lemma552}, it suffices to show that
\begin{lemma}\label{Lemmaxxx2}It holds that

\begin{itemize}
  \item $(1):$ for $\alpha\geq\frac{9}{8}$, $\theta_{yy}(\alpha,i)<\pi\alpha M_{yy}(\alpha,i)$;
  \item $(2):$ for $\alpha\in[1,\frac{9}{8}]$,
  $\partial_\alpha\big(\frac{\alpha M_y(\alpha,i)}{\theta_y(\alpha,i)}\big)>0$.
\end{itemize}

\end{lemma}

Now by the reduction(Proposition \ref{PropJb} and Lemma \ref{Lemmayyy}),
 we need only to show Lemmas \ref{Lemmaxxx} and \ref{Lemmaxxx2}.
 The proof of item $(1)$ in Lemma \ref{Lemmaxxx} is different to item $(2)$ in Lemma \ref{Lemmaxxx} and Lemma \ref{Lemmaxxx2}.
 We prove the former firstly then prove the latter.

 By Lemmas \ref{Lemma2} and \ref{Lemma3}, after some computations and regrouping, one gets
 \begin{lemma}\label{LemmaK1} An exponentially decaying expansion for $\Big(\theta_y(\alpha,iy)-\pi\alpha M_y(\alpha,iy)\Big)$.
\begin{equation}\aligned\nonumber
\theta_y(\alpha,iy)-\pi\alpha M_y(\alpha,iy)
=&\frac{1}{\sqrt{\frac{y}{\alpha}}\alpha^3}
\sum_{n\in \mathbb{Z}}
\Big(
\alpha^2(\pi^2 n^4 \alpha^2 y^2-2\pi\alpha y n^2+\frac{1}{4})\vartheta(\frac{y}{\alpha};0)\\
&-\alpha y\vartheta_X(\frac{y}{\alpha};0)-y^2\vartheta_{XX}(\frac{y}{\alpha};0)
\Big)\cdot e^{-\pi\alpha y n^2}.
\endaligned\end{equation}
\end{lemma}

Based on the expansion in Lemma \ref{LemmaK1}, for convenience, we denote that
\begin{equation}\aligned\label{RRR}
R_n(\alpha,y):=
\Big(\alpha^2(\pi^2 n^4 \alpha^2 y^2-2\pi\alpha y n^2+\frac{1}{4})\vartheta(\frac{y}{\alpha};0)
-\alpha y\vartheta_X(\frac{y}{\alpha};0)-y^2\vartheta_{XX}(\frac{y}{\alpha};0)\Big).
\endaligned\end{equation}

Then one has the summation
\begin{equation}\aligned\label{LLL}
\theta_y(\alpha,iy)-\pi\alpha M_y(\alpha,iy)
=&\frac{1}{\sqrt{\frac{y}{\alpha}}\alpha^3}
\sum_{n\in \mathbb{Z}}
R_n(\alpha,y)\cdot e^{-\pi\alpha y n^2}
\endaligned\end{equation}
Therefore, item $(1)$ in Lemma \ref{Lemmaxxx} is proved by \eqref{LLL} and the following
straightforward observation.

\begin{lemma}\label{LemmaK2} Assume that $\alpha\geq\frac{101}{100}$. Then
\begin{equation}\aligned\nonumber
R_n(\alpha,y)>0\;\;\hbox{for}\;\;\frac{y}{\alpha}\geq2,\;\;\forall n\in\mathbb{Z}.
\endaligned\end{equation}
\end{lemma}

 In fact, from the explicit expressions of $\vartheta(\frac{y}{\alpha};0), \vartheta_X(\frac{y}{\alpha};0), \vartheta_{XX}(\frac{y}{\alpha};0)$(See \eqref{TXY} and its variants), one has the following immediately,
 $$
 \vartheta(\frac{y}{\alpha};0)\geq1,\;\; \vartheta_X(\frac{y}{\alpha};0)\leq0, \;\;-\vartheta_{XX}(\frac{y}{\alpha};0)\leq4\pi^2 e^{-\pi\frac{y}{\alpha}}  (\hbox{if}\;\;\frac{y}{\alpha}\geq2).
 $$
 These estimates above will lead to Lemma \ref{LemmaK2} through \eqref{RRR}.

It remains to prove item $(2)$ in Lemma \ref{Lemmaxxx} and Lemma \ref{Lemmaxxx2}. We use an alternative expression of
$\Big(\theta_y(\alpha,iy)-\pi\alpha M_y(\alpha,iy)\Big)$. Before going to the proof, we introduce some notations.
We denote that
\begin{equation}\aligned\label{Xab}
X_a:&=X_a(\alpha,y):=
\sum_{n,m}(n^2-\frac{m^2}{y^2}) e^{-\pi\alpha (yn^2+\frac{m^2}{y})}
\\
X_b:&=X_b(\alpha,y):={\sum_{n,m}(yn^4-\frac{m^4}{y^3}) e^{-\pi\alpha (yn^2+\frac{m^2}{y})}}.
\endaligned\end{equation}

Then
\begin{equation}\aligned\label{Xab0}
\partial_y X_a:&=\partial_yX_a(\alpha,y):=
-\sum_{n,m}(\pi\alpha(n^2-\frac{m^2}{y^2})^2-\frac{2m^2}{y^3}) e^{-\pi\alpha (yn^2+\frac{m^2}{y})}
\\
\partial_yX_b:&=\partial_yX_b(\alpha,y):=-{\sum_{n,m}(\pi\alpha y(n^4-\frac{m^4}{y^4})(n^2-\frac{m^2}{y^2})-(n^4+\frac{3m^4}{y^4})) e^{-\pi\alpha (yn^2+\frac{m^2}{y})}}.
\endaligned\end{equation}

Therefore, one has the following simplifications immediately(Lemmas \ref{Lemma56} and \ref{Lemma57}).
\begin{lemma} A relation between $\theta(\alpha, iy), M(\alpha,iy)$ and $X_a, X_b$, and its variant.
\label{Lemma56}
\begin{equation}\aligned\nonumber
\theta_y(\alpha, iy)&=-\pi\alpha X_a,\;\; \pi\alpha M_y(\alpha,iy)=\pi\alpha X_a-(\pi\alpha)^2 X_b;\\
\theta_{yy}(\alpha, iy)&=-\pi\alpha\partial_y X_a,\;\; \pi\alpha M_{yy}(\alpha,iy)=\pi\alpha \partial_yX_a-(\pi\alpha)^2 \partial_yX_b.
\endaligned\end{equation}

\end{lemma}

\begin{lemma}A relation between $\frac{M_y(\alpha, iy)}{\theta_y(\alpha,iy)}$ and $\frac{X_b}{X_a}$, and its variant.
\label{Lemma57}
\begin{equation}\aligned\nonumber
\frac{M_y(\alpha, iy)}{\theta_y(\alpha,iy)}=
-\frac{1}{\pi\alpha}+\frac{X_b}{X_a},\;\;
\frac{M_{yy}(\alpha, iy)}{\theta_{yy}(\alpha,iy)}=
-\frac{1}{\pi\alpha}+\frac{\partial_yX_b}{\partial_yX_a}.
\endaligned\end{equation}

\end{lemma}

With these simplifications(Lemmas \ref{Lemma56} and \ref{Lemma57}),
item $(2)$ in Lemma \ref{Lemmaxxx} and Lemma \ref{Lemmaxxx2} are reformed in the following two lemmas(based on Lemma \ref{Lemma512}).

\begin{lemma}[=item $(2)$ in Lemma \ref{Lemmaxxx}]\label{Lemma510}
Assume that $\alpha\geq\frac{101}{100}$, then
$
\partial_y\frac{X_b}{X_a}\leq0\;\;\hbox{for}\;\; y\in[1,2\alpha].
$

\end{lemma}

\begin{lemma}[=Lemma \ref{Lemmaxxx2}]\label{Lemma511}
It holds that

\begin{itemize}
  \item $(1):$ for $\alpha\geq\frac{9}{8}$, $\frac{2}{\pi\alpha}\frac{\partial_yX_a\mid_{y=1}}
{\partial_yX_b\mid_{y=1}}<1$;
  \item $(2):$ for $\alpha\in[1,\frac{9}{8}]$,
  $\frac{1}{\alpha}\frac{\partial_yX_a\mid_{y=1}}
{\partial_yX_b\mid_{y=1}}<0$.
\end{itemize}
\end{lemma}

\begin{lemma}\label{Lemma512}
Assume that $\alpha\geq1$, then
$
\partial_y X_b\mid_{y=1}<0.
$
\end{lemma}

To prove Lemma \ref{Lemma510}, we have to(since $X_b\mid_{y=1}=X_a\mid_{y=1}=0$) split it into two cases, i.e., near $y=1$ and away from $y=1$.

\begin{lemma}\label{Lemma513}
Assume that $\alpha\geq\frac{101}{100}$. Then

\begin{itemize}
  \item $(1):$ for $y\in[1,\frac{11}{10}\alpha]$, $\partial_y\frac{\partial_y X_b}{\partial_y X_a}<0$;
  \item $(2):$ for $y\in[\frac{11}{10}\alpha, 2\alpha]$,
  $\partial_y\frac{X_b}{X_a}<0$.
\end{itemize}
\end{lemma}
Note that (since $X_b\mid_{y=1}=X_a\mid_{y=1}=0$), by a monotonicity rule(see e.g. \cite{And}), item $(1)$ in Lemma \ref{Lemma513} implies that $\partial_y\frac{X_b}{X_a}<0$ for $y\in[1,\frac{11}{10}\alpha]$, thus Lemma \ref{Lemma513} implies Lemma \ref{Lemma510}.
Now we give the outline of the proof of Lemmas \ref{Lemma512} and \ref{Lemma513}.

Recall the expressions in \eqref{Xab}, we further separate

\begin{equation}\aligned\label{Xab2}
\mathcal{A}_a:&=\mathcal{A}_a(\alpha,y):=
\sum_{n^2+m^2\leq2}(n^2-\frac{m^2}{y^2}) e^{-\pi\alpha (yn^2+\frac{m^2}{y})},\\
 \mathcal{A}_e:&=\mathcal{A}_e(\alpha,y):=
\sum_{n^2+m^2\geq3}(n^2-\frac{m^2}{y^2}) e^{-\pi\alpha (yn^2+\frac{m^2}{y})}
\endaligned\end{equation}
being the dominated and error terms of $X_a$.
Then $
X_a=\mathcal{A}_a+\mathcal{A}_e.
$
Similarly for $X_b$, one separates that
\begin{equation}\aligned\label{Xab2}
\mathcal{B}_a:&=\mathcal{B}_a(\alpha,y):={\sum_{n^2+m^2\leq2}(yn^4-\frac{m^4}{y^3}) e^{-\pi\alpha (yn^2+\frac{m^2}{y})}}\\
\mathcal{B}_e:&=\mathcal{B}_a(\alpha,y):={\sum_{n^2+m^2\geq3}(yn^4-\frac{m^4}{y^3}) e^{-\pi\alpha (yn^2+\frac{m^2}{y})}}.
\endaligned\end{equation}
being the the dominated and error terms of $X_b$, and then
$
X_b=\mathcal{B}_a+\mathcal{B}_e.
$

Explicitly, one has
\begin{equation}\aligned\label{Xab4}
\mathcal{A}_a&=2e^{-\pi\alpha y}-\frac{2}{y^2}e^{-\frac{\pi\alpha}{y}}+4(1-\frac{1}{y^2})e^{-\pi\alpha(y+\frac{1}{y})},\\
\mathcal{B}_a&=2ye^{-\pi\alpha y}-\frac{2}{y^3}e^{-\frac{\pi\alpha}{y}}+4(y-\frac{1}{y^3})e^{-\pi\alpha(y+\frac{1}{y})}.
\endaligned\end{equation}

$\mathcal{A}_e, \mathcal{B}_e$ are the error terms of $X_a, X_b$ respectively, indeed, one has

\begin{equation}\aligned\label{Xab5}
\mid\frac{\mathcal{A}_e}{\mathcal{A}_a}\mid&=O(e^{-\frac{3\pi\alpha}{y}}), \;\;\mid\frac{\mathcal{B}_e}{\mathcal{B}_a}\mid=O(e^{-\frac{3\pi\alpha}{y}})\;\;\hbox{for}\;\;y\in[\frac{11}{10},2\alpha],\;\alpha\geq1,\\
\mid\frac{\partial_y\mathcal{A}_e}{\partial_y\mathcal{A}_a}\mid&=O(e^{-\frac{3\pi\alpha}{y}}), \;\;\mid\frac{\partial_y\mathcal{B}_e}{\partial_y\mathcal{B}_a}\mid=O(e^{-\frac{3\pi\alpha}{y}})
\;\;\hbox{for}\;\;y\in[1,\frac{11}{10}\alpha],\;\alpha\geq1.
\endaligned\end{equation}

At this stage, we shall state a monotonicity property on $\frac{\mathcal{B}_a}{\mathcal{A}_a}$ and $\frac{\partial_y\mathcal{B}_a}{\partial_y\mathcal{A}_a}$ based on the explicit expressions in \eqref{Xab4}.
The proof is by straightforward and tedious computations, we omit the detail here.

\begin{lemma}\label{Lemma514}
Assume that $\alpha\geq\frac{101}{100}$. Then

\begin{itemize}
  \item $(1):$ for $y\in[1,\frac{11}{10}\alpha]$, $\partial_y\frac{\partial_y\mathcal{B}_a}{\partial_y\mathcal{A}_a}<0$ or equivalently,
  $\partial_{yy}\mathcal{B}_a\cdot \partial_{y}\mathcal{A}_a-\partial_{y}\mathcal{B}_a\cdot \partial_{yy}\mathcal{A}_a<0$;
  \item $(2):$ for $y\in[\frac{11}{10}\alpha, 2\alpha]$, $\partial_y\frac{\mathcal{B}_a}{\mathcal{A}_a}$
   or equivalently,
  $\partial_{y}\mathcal{B}_a\cdot \mathcal{A}_a-\mathcal{B}_a\cdot \partial_{y}\mathcal{A}_a<0$.
\end{itemize}
\end{lemma}
Lemma \ref{Lemma514} is very close to Lemma \ref{Lemma510}, up to a controllable small term. We shall see this in the proof.
We give the proof of item $(2)$ in Lemma \ref{Lemma513} firstly. A direct calculation shows that

\begin{equation}\aligned\label{Xab6}
\partial_y\frac{X_b}{X_a}=\frac{1}{X_a^2}\Big(
\partial_y X_b\cdot X_a-\partial_y X_a\cdot X_b\Big)
=\frac{1}{X_a^2}
\Big(
\partial_y \mathcal{B}_a\cdot \mathcal{A}_a-\partial_y \mathcal{A}_a\cdot \mathcal{B}_a
+\epsilon_{\uppercase\expandafter{\romannumeral1}}
\Big),
\endaligned\end{equation}
where explicitly, the error term $\epsilon_{\uppercase\expandafter{\romannumeral1}}$ is denoted by
\begin{equation}\aligned\nonumber
\epsilon_{\uppercase\expandafter{\romannumeral1}}
=(\partial_y\mathcal{B}_a\mathcal{A}_e+\partial_y\mathcal{B}_e\mathcal{A}_a+\partial_y\mathcal{B}_e\mathcal{A}_e)
-(\mathcal{B}_a\partial_y\mathcal{A}_e+\mathcal{B}_e\partial_y\mathcal{A}_a+\mathcal{B}_e\partial_y\mathcal{A}_e).
\endaligned\end{equation}
Comparing to the major term $\Big(\partial_y \mathcal{B}_a\cdot \mathcal{A}_a-\partial_y \mathcal{A}_a\cdot \mathcal{B}_a\Big)$, the error term
$\epsilon_{\uppercase\expandafter{\romannumeral1}}$ is relative small(each of them has one of $\mathcal{A}_e, \mathcal{B}_e, \partial_y\mathcal{A}_e,\partial_y\mathcal{B}_e$(see \eqref{Xab5})), it is straightforward but tedious to get that
\begin{equation}\aligned\label{Xab7}
|\frac{\epsilon_{\uppercase\expandafter{\romannumeral1}}}
{\partial_y \mathcal{B}_a\cdot \mathcal{A}_a-\partial_y \mathcal{A}_a\cdot \mathcal{B}_a}|
\leq\frac{1}{2}\;\;\hbox{for}\;\;y\in[\frac{11}{10}\alpha, 2\alpha],\;\alpha\geq\frac{101}{100}.
\endaligned\end{equation}
The upper bound $\frac{1}{2}$ in \eqref{Xab7} is very rough, actually it is much smaller, however the bound is enough for our proof.
Therefore, \eqref{Xab6}, \eqref{Xab7} and  item $(2)$ in Lemma \ref{Lemma514} lead to item $(2)$ in Lemma \ref{Lemma513}.

Similar steps and estimates work for $\partial_y\frac{\partial_yX_b}{\partial_yX_a}$.

\begin{equation}\aligned\label{Xab8}
\partial_y\frac{\partial_yX_b}{\partial_yX_a}=\frac{1}{\partial_yX_a^2}\Big(
\partial_{yy}X_b\cdot  \partial_yX_a-\partial_{yy} X_a\cdot \partial_yX_b\Big)
=\frac{1}{ \partial_yX_a^2}
\Big(
\partial_{yy} \mathcal{B}_a\cdot \partial_y\mathcal{A}_a-\partial_{yy} \mathcal{A}_a\cdot \partial_y\mathcal{B}_a
+\epsilon_{\uppercase\expandafter{\romannumeral2}}
\Big),
\endaligned\end{equation}
where explicitly,
\begin{equation}\aligned\nonumber
\epsilon_{\uppercase\expandafter{\romannumeral2}}
=(\partial_{yy}\mathcal{B}_a\partial_y\mathcal{A}_e+\partial_{yy}\mathcal{B}_e\partial_y\mathcal{A}_a
+\partial_{yy}\mathcal{B}_e\partial_y\mathcal{A}_e)
-(\partial_y\mathcal{B}_a\partial_{yy}\mathcal{A}_e+\partial_y\mathcal{B}_e\partial_{yy}\mathcal{A}_a
+\partial_y\mathcal{B}_e\partial_{yy}\mathcal{A}_e).
\endaligned\end{equation}
And the remainder estimate holds
\begin{equation}\aligned\label{Xab9}
|\frac{\epsilon_{\uppercase\expandafter{\romannumeral2}}}
{\partial_{yy}X_b\cdot  \partial_yX_a-\partial_{yy} X_a\cdot \partial_yX_b}|
\leq\frac{1}{2}\;\;\hbox{for}\;\;y\in[1,\frac{11}{10}\alpha],\;\alpha\geq\frac{101}{100}.
\endaligned\end{equation}
Therefore, \eqref{Xab8}, \eqref{Xab9} and  item $(1)$ in Lemma \ref{Lemma514} lead to item $(1)$ in Lemma \ref{Lemma513}.

At the end of this Section, we give the proof of Lemmas \ref{Lemma511} and \ref{Lemma512}. For $y=1$, one has the clean form(by \eqref{Xab0})
\begin{equation}\aligned\nonumber
\partial_yX_a\mid_{y=1}&=-\sum_{n,m}(\pi\alpha(n^2-m^2)^2-2m^2)e^{-\pi\alpha(n^2+m^2)}\\
\partial_yX_b\mid_{y=1}&=-
\sum_{n,m}(\pi\alpha(n^4-m^4)(n^2-m^2)-(n^4+3m^4))e^{-\pi\alpha(n^2+m^2)}
.
\endaligned\end{equation}
Before going to the proof, we do some preparing work. As the same decomposition in \eqref{Xab2}, we denote that
\begin{equation}\aligned\label{AC}
A:&=A(\alpha)=\sum_{n^2+m^2\leq2}(\pi\alpha(n^2-m^2)^2-2m^2)e^{-\pi\alpha(n^2+m^2)},\\
C:&=C(\alpha)=\sum_{n^2+m^2\geq3}(\pi\alpha(n^2-m^2)^2-2m^2)e^{-\pi\alpha(n^2+m^2)}.
\endaligned\end{equation}

And similarly,
\begin{equation}\aligned\label{BD}
B:&=B(\alpha)=\sum_{n^2+m^2\leq2}(\pi\alpha(n^4-m^4)(n^2-m^2)-(n^4+3m^4))e^{-\pi\alpha(n^2+m^2)},\\
D:&=D(\alpha)=\sum_{n^2+m^2\geq3}(\pi\alpha(n^4-m^4)(n^2-m^2)-(n^4+3m^4))e^{-\pi\alpha(n^2+m^2)}.
\endaligned\end{equation}
Explicitly, one has
\begin{equation}\aligned\nonumber
A=A(\alpha)=4e^{-\pi\alpha}(\pi\alpha-1-2e^{-\pi\alpha}),\;\;
B=B(\alpha)=4e^{-\pi\alpha}(\pi\alpha-2-4e^{-\pi\alpha}).
\endaligned\end{equation}

Then
\begin{equation}\aligned\nonumber
\partial_yX_a\mid_{y=1}=-(A+C),\;\partial_yX_b\mid_{y=1}=-(B+D),\;
\frac{\partial_yX_a\mid_{y=1}}
{\partial_yX_b\mid_{y=1}}=\frac{A+C}{B+D}.
\endaligned\end{equation}
One also notes that $\frac{C}{A}=O(e^{-3\pi\alpha}),\;\frac{D}{B}=O(e^{-3\pi\alpha})$

For the proof of Lemma \ref{Lemma512}, one needs only to note that $B>0$ and $D>0$ for $\alpha\geq1$.

For the proof of item $(1)$ in Lemma \ref{Lemma511}, one needs only the following upper bound estimate and some simple algebra.
\begin{lemma}\label{Lemma515}
Assume that $\alpha\geq1$, then
\begin{equation}\aligned\nonumber
\frac{\partial_yX_a\mid_{y=1}}
{\partial_yX_b\mid_{y=1}}
\leq \frac{A}{B}=\frac{\pi\alpha-1-2e^{-\pi\alpha}}{\pi\alpha-2-4 e^{-\pi\alpha}}.
\endaligned\end{equation}
\end{lemma}

The upper bound estimate in Lemma \ref{Lemma515} is equivalent to
\begin{equation}\aligned\label{ABCD}
\frac{A+C}{B+D}<\frac{A}{B}.
\endaligned\end{equation}
Inequality \eqref{ABCD} is equivalent to
\begin{equation}\aligned\label{ABCDa}
\frac{C}{D}<\frac{A}{B}.
\endaligned\end{equation}
\eqref{ABCDa} can be easily verified by \eqref{AC} and \eqref{BD}. Indeed, one has
\begin{equation}\aligned\nonumber
\frac{C}{D}<1<\frac{A}{B}.
\endaligned\end{equation}

For the proof of item $(1)$ in Lemma \ref{Lemma511}, one has
\begin{equation}\aligned\nonumber
\frac{1}{\alpha}\frac{\partial_yX_a\mid_{y=1}}
{\partial_yX_b\mid_{y=1}}=\frac{1}{\alpha}\cdot\frac{A}{B}\cdot\frac{1+\frac{C}{A}}{1+\frac{D}{B}}.
\endaligned\end{equation}
See $\frac{A}{B}$ in Lemma \ref{Lemma515}, and note that $\frac{1+\frac{C}{A}}{1+\frac{D}{B}}\approx1$(if $\alpha\geq1$).
A direct calculation shows that
\begin{equation}\aligned\nonumber
-\partial_\alpha\Big(\frac{1}{\alpha}\frac{\partial_yX_a\mid_{y=1}}
{\partial_yX_b\mid_{y=1}}\Big)&=\frac{1}{\alpha}\cdot\frac{1}{1+\frac{D}{B}}
\Big(
\frac{1}{\alpha}(1+\frac{1+2e^{-\pi\alpha}}{\pi\alpha-2-4e^{-\pi\alpha}})(1+\frac{C}{A})
+\pi\frac{1+2(\pi\alpha+1)e^{-\pi\alpha}}{(\pi\alpha-2-4e^{-\pi\alpha})^2}(1+\frac{C}{A})\\
&+(1+\frac{1+2e^{-\pi\alpha}}{\pi\alpha-2-4e^{-\pi\alpha}})\cdot\big(
(\frac{C}{A})'-(\frac{D}{B})'\frac{1+\frac{C}{A}}{1+\frac{D}{B}}
\big)
\Big)\\
&\geq
\frac{1}{\alpha}\cdot\frac{1}{1+\frac{D}{B}}
\Big(
\frac{1}{\alpha}+\frac{1}{\pi\alpha}
+(1+\frac{1+2e^{-\pi\alpha}}{\pi\alpha-2-4e^{-\pi\alpha}})\cdot\big(
(\frac{C}{A})'-(\frac{D}{B})'\frac{1+\frac{C}{A}}{1+\frac{D}{B}}
\big)
\Big)\\
&=
\frac{1}{\alpha}\cdot\frac{1}{1+\frac{D}{B}}
\Big(
\frac{1}{\alpha}+\frac{1}{\pi\alpha}
+O(e^{-3\pi\alpha})
\Big)>0
\endaligned\end{equation}
here $(\frac{C}{A})'=O(e^{-3\pi\alpha}), (\frac{D}{B})'=O(e^{-3\pi\alpha})$. And indeed
$
\mid\frac{\big(
(\frac{C}{A})'-(\frac{D}{B})'\frac{1+\frac{C}{A}}{1+\frac{D}{B}}
\big)}{\frac{1}{\alpha}}\mid\leq\frac{1}{2}
$ for $\alpha\in[1,\frac{9}{8}]$(with tedious but simple computations and hence omitted here).
These complete the proof.

\section{The analysis on $\Gamma_b$}

Recall that
\begin{equation}\aligned\nonumber
\Gamma_b=\{
z\in\mathbb{H}: z=e^{i\theta},\; \theta\in[\frac{\pi}{3},\frac{\pi}{2}]
\}.
\endaligned\end{equation}
In this section, we aim to establish that
\begin{theorem}\label{Th41} Assume that $\alpha\in(0,1]$, then
\begin{equation}\aligned
\min_{z\in\Gamma_b}M(\alpha,z)=
\begin{cases}
\hbox{is achieved at}\;\;i, &\hbox{if}\;\; \alpha\in(0,\alpha_2),\\
\hbox{is achieved at}\;\;i\;\hbox{or}\;e^{i\frac{\pi}{3}} &\hbox{if}\;\; \alpha=\alpha_2,\\
\hbox{is achieved at}\;\;e^{i\frac{\pi}{3}}, &\hbox{if}\;\; \alpha\in(\alpha_2,1].
\end{cases}
\endaligned\end{equation}
Here $\alpha_2=0.9203340937\cdots$. Further,
${\alpha_2}$ is the unique solution of
\begin{equation}\aligned\nonumber
M(\alpha,i)=M(\alpha,e^{i\frac{\pi}{3}}),\;\;\hbox{for}\;\;\alpha\in[\frac{5}{6},1].
\endaligned\end{equation}

\end{theorem}
\begin{remark}
It is interesting to note that for all $\alpha\in(0,1]$,
\begin{equation}\aligned
\min_{z\in\Gamma_b}M(\alpha,z)\;\;\hbox{is achieved at}\;\;i\;\;\hbox{or}\;\;e^{i\frac{\pi}{3}}.
\endaligned\end{equation}
\end{remark}
The direct analysis on the arc $\Gamma_b$ is probably involved, we transfer the analysis on the arc $\Gamma_b$
to a straight vertical interval. This is done by the next lemma and it is followed by conformal invariance(Lemma \ref{G111}).
\begin{lemma}[{\bf From the arc $\Gamma_b$ to the $\frac{1}{2}-$axis}]\label{LemmaH1} For all $\alpha>0$,
\begin{equation}\aligned\nonumber
M(\alpha,u+i\sqrt{1-u^2})
=M(\alpha,\frac{1}{2}+\frac{i}{2}\sqrt{\frac{1+u}{1-u}}).
\endaligned\end{equation}
\end{lemma}

By Lemma \ref{LemmaH1}, Theorem \ref{Th41} is equivalent to the following

\begin{theorem}\label{Th41a} Assume that $\alpha\in(0,1]$, then
\begin{equation}\aligned
\min_{y\in[\frac{1}{2},\frac{\sqrt3}{2}]}M(\alpha,\frac{1}{2}+iy)=
\begin{cases}
\hbox{is achieved at}\;\;\frac{1}{2}, &\hbox{if}\;\; \alpha\in(0,\alpha_2),\\
\hbox{is achieved at}\;\;\frac{1}{2}\;\hbox{or}\;\frac{\sqrt3}{2} &\hbox{if}\;\; \alpha=\alpha_2,\\
\hbox{is achieved at}\;\;\frac{\sqrt3}{2}, &\hbox{if}\;\; \alpha\in(\alpha_2,1].
\end{cases}
\endaligned\end{equation}
Here $\alpha_2=0.9203340937\cdots$.

\end{theorem}

In Theorem \ref{Th41a}, we have reduced to the analysis on the arc $\Gamma_b$ to an interval. Another difficulty here is the parameter $\alpha$ has no positive lower bound. We shall get rid of this difficulty by a duality property(Lemma \ref{Lemma1}) of the parameter $\alpha$.

By the duality of the functionals(Lemma \ref{Lemma1}), Theorem \ref{Th41a} is then equivalent to
\begin{theorem}\label{Th42} Assume that $\alpha\geq1$, then
\begin{equation}\aligned
\min_{y\in[\frac{1}{2},\frac{\sqrt3}{2}]}\Big(
\theta(\alpha,\frac{1}{2}+iy)-\pi\alpha M(\alpha,\frac{1}{2}+iy)
\Big)=
\begin{cases}
\hbox{is achieved at}\;\;\frac{\sqrt3}{2}, &\hbox{if}\;\; \alpha\in[1,\frac{1}{\alpha_2}),\\
\hbox{is achieved at}\;\;\frac{1}{2}\;\hbox{or}\;\frac{\sqrt3}{2} &\hbox{if}\;\; \alpha=\frac{1}{\alpha_2},\\
\hbox{is achieved at}\;\;\frac{1}{2}, &\hbox{if}\;\; \alpha\in(\frac{1}{\alpha_2},\infty).
\end{cases}
\endaligned\end{equation}
The $\frac{1}{\alpha_2}=1.086561943\cdots$ is located as the same in Theorem \ref{Th41}.
\end{theorem}

\begin{figure}
\centering
 \includegraphics[scale=0.25]{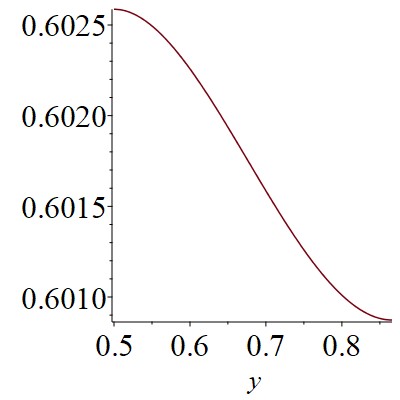}\includegraphics[scale=0.25]{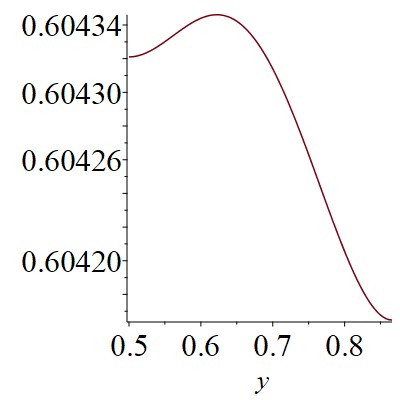}
 \includegraphics[scale=0.25]{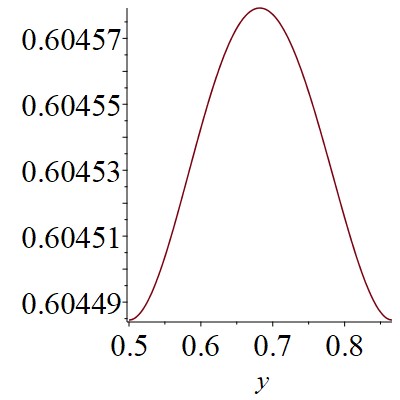}\includegraphics[scale=0.25]{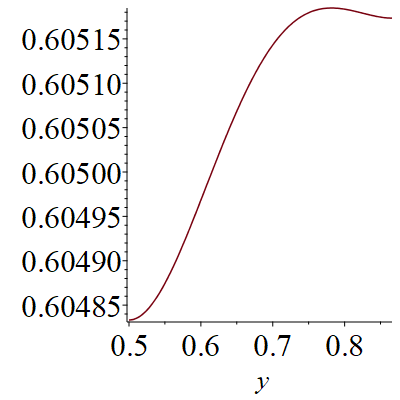}
 \includegraphics[scale=0.25]{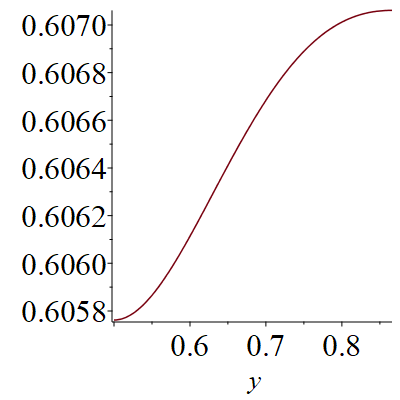}
 \caption{
 The shape of $\Big(
\theta(\alpha,\frac{1}{2}+iy)-\pi\alpha M(\alpha,\frac{1}{2}+iy)
\Big)$ on $y\in[\frac{1}{2},\frac{\sqrt3}{2}]$ for various $\alpha\geq1$.
 }
\label{Y-Shape}
\end{figure}

In the rest of this Section, we aim to prove Theorem \ref{Th42}. To simplify the notations and reveal the inner structures, we denote that
\begin{equation}\aligned\label{Nab}
Y_a:&=Y_a(\alpha;y):=\sum_{n,m}(n^2-\frac{(m+\frac{n}{2})^2}{y^2}) e^{-\pi\alpha(yn^2+\frac{(m+\frac{n}{2})^2}{y})},\\
Y_b:&=Y_b(\alpha;y):=\sum_{n,m}y(n^4-\frac{(m+\frac{n}{2})^4}{y^4}) e^{-\pi\alpha(yn^2+\frac{(m+\frac{n}{2})^2}{y})}.
\endaligned\end{equation}

Then a direct calculation shows that
\begin{lemma}[A relation between $\theta_y, M_y$ and $Y_a, Y_b$] It holds that
\begin{equation}\aligned\nonumber
\theta_y(\alpha,\frac{1}{2}+iy)=-\pi\alpha Y_a,\;\;\pi\alpha M_y(\alpha,\frac{1}{2}+iy)=\pi\alpha Y_a-(\pi\alpha)^2 Y_b.
\endaligned\end{equation}
\end{lemma}

By \cite{Mon1988}, the sign of $Y_a$ is nonnegative.
\begin{lemma} For $\alpha>0$, it holds that
\begin{equation}\aligned\nonumber
Y_a\geq0\;\;\hbox{for}\;\; y\in[\frac{1}{2},\frac{\sqrt3}{2}].
\endaligned\end{equation}
\end{lemma}

To prove Theorem \ref{Th42}, we use the deformation
\begin{equation}\aligned\label{Py}
\frac{\partial}{\partial y}\Big(
\theta(\alpha,\frac{1}{2}+iy)-\pi\alpha M(\alpha,\frac{1}{2}+iy)
\Big)
=-2\pi\alpha Y_a\cdot(1-\frac{\pi\alpha}{2}\cdot\frac{Y_b}{Y_a}).
\endaligned\end{equation}

Here we use the following notations:

We first have two universal critical points of $\Big(
\theta(\alpha,\frac{1}{2}+iy)-\pi\alpha M(\alpha,\frac{1}{2}+iy)
\Big)$ independent of $\alpha$. It is a consequence of Lemma \ref{Lemma3.3}.

\begin{lemma}\label{Lemma5.2} For $\alpha\geq1$, both $\theta_y(\alpha,\frac{1}{2}+iy), M_y(\alpha,\frac{1}{2}+iy)$ satisfy the following property
\begin{equation}\aligned
(\cdot)\leq0\;\;\hbox{for}\;\;y\in[\frac{1}{2},\frac{\sqrt3}{2}],
\endaligned\end{equation}
with $"="$ holds if only $y\in\{\frac{1}{2},\frac{\sqrt3}{2}\}$.
It follows that $\Big(
\theta(\alpha,\frac{1}{2}+iy)-\pi\alpha M(\alpha,\frac{1}{2}+iy)
\Big)$ has two critical points $y=\frac{1}{2},\frac{\sqrt3}{2}$ independent of $\alpha$.
\end{lemma}
We shall further investigate the critical point of $\Big(
\theta(\alpha,\frac{1}{2}+iy)-\pi\alpha M(\alpha,\frac{1}{2}+iy)
\Big)$. In fact, we prove that

\begin{proposition}\label{Prop5.1} Assume that $\alpha\geq1$. Then
on the interval $[\frac{1}{2},\frac{\sqrt3}{2}]$, the function $\Big(
\theta(\alpha,\frac{1}{2}+iy)-\pi\alpha M(\alpha,\frac{1}{2}+iy)
\Big)$ admits at most one critical point except $y=\frac{1}{2},\frac{\sqrt3}{2}$.
Furthermore, the possible additional critical point is a local maximum, and
\begin{equation}\aligned
\min_{y\in[\frac{1}{2},\frac{\sqrt3}{2}]}\Big(
\theta(\alpha,\frac{1}{2}+iy)-\pi\alpha M(\alpha,\frac{1}{2}+iy)
\Big)\;\;
\hbox{is achieved at}\;\;\frac{1}{2}\;\hbox{or}\;\frac{\sqrt3}{2}.
\endaligned\end{equation}
\end{proposition}

By Proposition \ref{Prop5.1}, all various kinds shapes of the function $\Big(
\theta(\alpha,\frac{1}{2}+iy)-\pi\alpha M(\alpha,\frac{1}{2}+iy)
\Big)$ on the interval $[\frac{1}{2},\frac{\sqrt3}{2}]$ are illustrated in Picture \ref{Y-Shape}. Generally speaking, it admits three kinds of shapes, {\bf $(1):$ decreasing; $(2):$ first increasing and then decreasing; $(3):$ increasing.}

By \eqref{Py}, to locate the possible additional critical point of $\Big(
\theta(\alpha,\frac{1}{2}+iy)-\pi\alpha M(\alpha,\frac{1}{2}+iy)
\Big)$, it suffices to solve the following equation
\begin{equation}\aligned\label{Yab}
\frac{\pi\alpha}{2}\cdot\frac{Y_b}{Y_a}=1, \;\;y\in[\frac{1}{2},\frac{\sqrt3}{2}].
\endaligned\end{equation}
Here $Y_a, Y_b$ are denoted in \eqref{Nab}.

We first show that \eqref{Yab} has no solution for $\alpha\geq\frac{6}{5}$, i.e., $\Big(
\theta(\alpha,\frac{1}{2}+iy)-\pi\alpha M(\alpha,\frac{1}{2}+iy)
\Big)$ has no other critical point on $[\frac{1}{2},\frac{\sqrt3}{2}]$ except $y=\frac{1}{2},\frac{\sqrt3}{2}$.
This is done by a lower bound estimate.
\begin{lemma}\label{Lemma5.3} For $\alpha\geq\frac{6}{5}$,
\begin{equation}\aligned
\frac{\pi\alpha}{2}\cdot\frac{Y_b}{Y_a}\geq\frac{11}{10}>1\;\;\hbox{for}\;\;y\in[\frac{1}{2},\frac{\sqrt3}{2}].
\endaligned\end{equation}
\end{lemma}
Lemma \ref{Lemma5.3} is proved by an approximation of $\frac{Y_b}{Y_a}$. Indeed, we have
\begin{lemma}\label{Lemma5.3a} For $\alpha\geq\frac{6}{5}$, it holds that
\begin{equation}\aligned
P(\alpha,y)-\frac{1}{20}\leq\frac{Y_b}{Y_a}\leq P(\alpha,y)\;\;\hbox{for}\;\;y\in[\frac{1}{2},\frac{\sqrt3}{2}].
\endaligned\end{equation}

Here
\begin{equation}\aligned
P(\alpha,y):=\frac{2y(1-\frac{1}{16y^4}) e^{-\pi\alpha(y-\frac{3}{4y})}-\frac{1}{y^3}+16y e^{-\pi\alpha(4y-\frac{1}{y})}}
{2(1-\frac{1}{4y^2}) e^{-\pi\alpha(y-\frac{3}{4y})}-\frac{1}{y^2}+4 e^{-\pi\alpha(4y-\frac{1}{y})}}.
\endaligned\end{equation}
\end{lemma}

We construct Lemma \ref{Lemma5.3a} by
a reformulation of $Y_a, Y_b$.
\begin{lemma}[New form of $Y_b$]\label{Lemma5.3b} It holds that
\begin{equation}\aligned
Y_a&=\sum_{p\equiv q(\mod2)}(p^2-\frac{q^2}{4y^2}) e^{-\pi\alpha (p^2y+\frac{q^2}{4y})},\\
Y_b&=\sum_{p\equiv q(\mod2)}y(p^4-\frac{q^4}{16y^4}) e^{-\pi\alpha (p^2y+\frac{q^2}{4y})}.
\endaligned\end{equation}
\end{lemma}
By the expressions in Lemma \ref{Lemma5.3b}. We then define
\begin{equation}\aligned
Y_{ap}:&=\sum_{(p,q)\in\{\pm(1,\pm1),\pm(2,0),\pm(0,2)\}}(p^2-\frac{q^2}{4y^2}) e^{-\pi\alpha (p^2y+\frac{q^2}{4y})},\\
Y_{bp}:&=\sum_{(p,q)\in\{\pm(1,\pm1),\pm(2,0),\pm(0,2)\}}y(p^4-\frac{q^4}{16y^4}) e^{-\pi\alpha (p^2y+\frac{q^2}{4y})}
\endaligned\end{equation}
be the approximate parts of $Y_a$ and $Y_b$ respectively. In fact, the approximation function $P(\alpha,y)$ in Lemma \ref{Lemma5.3a}
is given by
\begin{equation}\aligned
P(\alpha,y)=\frac{Y_{bp}}{Y_{ap}}.
\endaligned\end{equation}

Proceeding by Lemma \ref{Lemma5.3a}, Lemma \ref{Lemma5.3} is proved by
\begin{lemma}\label{Lemma5.3c} For $\alpha\geq\frac{6}{5}$,
\begin{equation}\aligned
\frac{\pi\alpha}{2}\cdot\frac{Y_b}{Y_a}\geq\frac{\pi\alpha}{2}\cdot \big(P(\alpha,y)-\frac{1}{20}\big)>\frac{11}{10}>1\;\;\hbox{for}\;\;y\in[\frac{1}{2},\frac{\sqrt3}{2}].
\endaligned\end{equation}
\end{lemma}

Next, we shall show that \eqref{Yab} admits at most one solution for $\alpha\in[1,\frac{6}{5}]$. This is proved by a
monotonicity property
\begin{lemma} \label{Lemma5.4}For $\alpha\in[1,\frac{6}{5}]$,
\begin{equation}\aligned
\frac{\partial}{\partial y}\frac{Y_b}{Y_a}\leq0\;\;\hbox{for}\;\;y\in[\frac{1}{2},\frac{\sqrt3}{2}].
\endaligned\end{equation}
\end{lemma}
The proof of Lemmas \ref{Lemma5.3c} and \ref{Lemma5.4} is similar to that of Lemma \ref{Lemma514}, we omit the details here.

By the deformation \eqref{Py}, Lemmas \ref{Lemma5.2}, \ref{Lemma5.3} and \ref{Lemma5.4} give the proof of Proposition \ref{Prop5.1}.
In Proposition \ref{Prop5.1}, we already obtain that
\begin{equation}\aligned
\min_{y\in[\frac{1}{2},\frac{\sqrt3}{2}]}\Big(
\theta(\alpha,\frac{1}{2}+iy)-\pi\alpha M(\alpha,\frac{1}{2}+iy)
\Big)\;\;
\hbox{is achieved at}\;\;\frac{1}{2}\;\hbox{or}\;\frac{\sqrt3}{2}.
\endaligned\end{equation}
By the deformation \eqref{Py} and Lemma \ref{Lemma5.3}, we have
\begin{equation}\aligned
\hbox{for}\;\;\alpha\geq\frac{6}{5}, \;\;\min_{y\in[\frac{1}{2},\frac{\sqrt3}{2}]}\Big(
\theta(\alpha,\frac{1}{2}+iy)-\pi\alpha M(\alpha,\frac{1}{2}+iy)
\Big)\;\;
\hbox{is achieved at}\;\;\frac{\sqrt3}{2}.
\endaligned\end{equation}
For $\alpha\in[1,\frac{6}{5}]$, we shall further determine where the minimizer is $\frac{1}{2}$ and where the minimizer is $\frac{\sqrt3}{2}$.
This is classified by a comparison lemma.

\begin{lemma}[A comparison between the values on $\frac{1}{2}$ and $\frac{\sqrt3}{2}$]\label{Lemma5.5}
For $\alpha\in[1,\frac{6}{5}]$, then
\begin{equation}\aligned
\Big(
\theta(\alpha,\frac{1}{2}+iy)-\pi\alpha M(\alpha,\frac{1}{2}+iy)
\Big)\mid_{y=\frac{1}{2}}&\leq\Big(
\theta(\alpha,\frac{1}{2}+iy)-\pi\alpha M(\alpha,\frac{1}{2}+iy)
\Big)\mid_{y=\frac{\sqrt3}{2}}\\
&\Leftrightarrow\alpha\leq1.086561943\cdots.
\endaligned\end{equation}
\end{lemma}
By Lemma \ref{Lemma5.5} and Proposition \ref{Prop5.1}, the proof of Theorem \ref{Th42} is complete.
An alternate and equivalent version of the comparison in Lemma \ref{Lemma5.5}(by the duality in Lemma \ref{LemmaH1}) is
\begin{lemma}
For $\alpha\in[\frac{5}{6},1]$, then
\begin{equation}\aligned
M(\alpha,i)\leq M(\alpha,e^{i\frac{\pi}{3}})
&\Leftrightarrow\alpha\leq0.9203340937\cdots.
\endaligned\end{equation}

\end{lemma}

\section{Proof of Theorem \ref{Th1}}

We are ready to prove our main result.

By \cite{LuoA}, we already prove that
\begin{equation}\aligned\nonumber
\hbox{for}\;\;\alpha\geq1\;\;\min_{z\in\mathbb{H}}\Big(\pi\alpha M(\alpha,z)-\frac{1}{2}\theta(\alpha,z)\Big)\;\;\hbox{is achieved at}\;\;i.
\endaligned\end{equation}
This, in particular, implies that
\begin{equation}\aligned\nonumber
\hbox{for}\;\;\alpha\geq1\;\;\min_{z\in\mathbb{H}}M(\alpha,z)\;\;\hbox{is achieved at}\;\;i.
\endaligned\end{equation}
The cases of $\alpha\in(0,1)$ is much more complicated. This is partially due to the slow convergence of the double series
in $M(\alpha,z)$. We actually invert it into the cases $\alpha\geq1$ with a new functional, namely, the following two minimization problems are equivalent
\begin{equation}\aligned\nonumber
\min_{z\in\mathbb{H}}M(\alpha,z)\;\;\hbox{for}\;\;\alpha\in(0,1]\;\;
\Leftrightarrow\;\; \min_{z\in\mathbb{H}}\Big(\theta(\alpha,z)-\pi\alpha M(\alpha,z)\Big)\;\;\hbox{for}\;\;\alpha\in[1,\infty).
\endaligned\end{equation}
This is achieved by an observed duality relation
\begin{equation}\aligned\nonumber
M(\frac{1}{\alpha},z)=\frac{\alpha^2}{\pi}\Big(
\theta(\alpha,z)-\pi\alpha M(\alpha,z)
\Big).
\endaligned\end{equation}
in Lemma \ref{Lemma1}.
Thanks to this duality relation, by Propositions \ref{Prop1} and \ref{PropA1}, we obtain that
\begin{equation}\aligned\label{L100}
\min_{z\in\mathbb{H}}M(\alpha,z)=\min_{z\in\Gamma}M(\alpha,z)\;\;\hbox{for}\;\;\alpha\in(0,1],
\endaligned\end{equation}
here $\Gamma=\Gamma_a\cup\Gamma_b$. \eqref{L100} reduces the location of the minimizers onto a curve $\Gamma$.
The minimizers of $M(\alpha,z)$ on $\Gamma_a$ and $\Gamma_b$ for $\alpha\in(0,1]$ were classified by Theorems \ref{Th31} and \ref{Th41} respectively.
Theorem \ref{Th1} then follows by \eqref{L100} and Theorems \ref{Th31} and \ref{Th41}.

\bigskip
\noindent
{\bf Acknowledgements.}
S. Luo thanks Prof. Y.Y. Hu(Central South University) for useful discussions.
S. Luo is grateful to Prof. W. M. Zou(Tsinghua university) and Prof. H. J. Zhao(Wuhan University) for their constant support and encouragement.
 The research of S. Luo is partially supported by NSFC(Nos. 12261045, 12001253) and double thousands plan of Jiangxi(jxsq2019101048). The research of J. Wei is partially supported by NSERC of Canada.

{\bf Statements and Declarations: there is no conflict of interest.}

{\bf Data availability: the manuscript has no associated data.}
\bigskip


\end{document}